\documentclass[dvipdf]{amsart}

\title[]{Extremal length boundary of Teichm\"uller space contains non-Busemann points}
\author{Hideki Miyachi}
\address{Department of Mathematics,
Graduate School of Science,
Osaka University,　
Machikaneyama 1-1,
Toyonaka,
Osaka 560-0043, Japan}
\thanks{The author is partially supported by the Ministry of Education, Science, Sports and Culture, Grant-in-Aid for Scientific Research (C), 21540177}

\makeatletter
\@addtoreset{equation}{section}

\makeatother

\usepackage{amsfonts, amsmath, amssymb, amsthm}
\usepackage{url}
\usepackage[dvipdf]{graphicx}
\usepackage[dvipdf]{color}

\newtheorem{theorem}{Theorem}[section]
\newtheorem{lemma}{Lemma}[section]
\newtheorem{proposition}{Proposition}[section]
\newtheorem{corollary}{Corollary}[section]

\newtheorem{claim}{Claim}

\newcommand{\ext}{{\rm Ext}}

\begin{document}
\maketitle

\begin{abstract}
We present an overview of the extremal length embedding of Teichm\"uller space and its extremal length compactification. For Teichm\"uller spaces of dimension at least two, we describe a large class of non-Busemann points on the metric boundary, that is, points that cannot be realized as limits of almost geodesic rays.
%In this paper,
%we shall show that the metric boundary of the Teichm\"uller space
%with respect to the Teichm\"uller distance
%contains non-Busemann points
%when the complex dimension of the Teichm\"uller space
%is at least two.
\end{abstract}

\section{Introduction}

\subsection{Background}
Let $M=(M,\rho)$ be a locally compact complete metric space.
In \cite{Rieffel}
M. Rieffel defined the metric compactification of $M$ and the metric boundary
to be the new points added to $M$ in the compactification.
This metric compactification with base point $x_0\in M$
is the maximal ideal space of the $C^*$-algebra
generated by
constant functions,
continuous functions vanishing at infinity,
and continuous functions which form
$$
%\varphi_y(x)=\rho (x,x_0)-\rho (x,y)
M\ni x\mapsto \rho (x,x_0)-\rho (x,y)
$$
indexed by $y\in M$.
He observed that the metric compactification can be naturally identified with a compactification given by M. Gromov in \cite{Gromov}.
In \S4 of \cite{Rieffel}
Rieffel and also in \S8.2
of Chapter II of \cite{BriF}
Bridson and Haefliger call this compactification the
\emph{horofunction compactification}.
%
%
%The \emph{metric boundary} of $M$
%was defined by M. Rieffel in \cite{Rieffel}
%as the boundary of a \emph{metric compactification} of $M$.
%%although this boundary was first introduced by M. Gromov
%%in 1978
%%in a slight different way (\cite{Gromov}).
%The metric compactification of $M$
%with basepoint $x_0\in M$
%is
%the maximal ideal space of the $C^*$-algebra
%generated by
%constant functions,
%continuous functions vanishing at infinity,
%and continuous functions which form
%$$
%%\varphi_y(x)=\rho (x,x_0)-\rho (x,y)
%M\ni x\mapsto \rho (x,x_0)-\rho (x,y)
%$$
%indexed by $y\in M$.
%He observed that the metric compactification is
%naturally identified with the compactification
%given by M. Gromov
%in \cite{Gromov},
%which is recently called the \emph{horofunction compactification}
%%and is determined independently
%%of the choice of base points
%(cf. \S4 in \cite{Rieffel} and \S8.12 of Chapter II in \cite{BriF}).
%% and \S\ref{subsec:metricboundary}).

In \cite{Rieffel},
M. Rieffel also defined geodesic-like sequences (or rays)
in a pointed metric space
which he called
\emph{almost geodesic rays}
(cf. \S\ref{subsec:almost_geodesics}).
He observed that any almost geodesic ray converges to a point in the metric boundary.
A point in the metric boundary
is called a \emph{Busemann point}
if it is the limit point of an almost geodesic ray.
Rieffel also asks whether every point in the metric boundary of a given metric space is a Busemann point (see the paragraph after Definition 4.8 in \cite{Rieffel}).
Related to this problem,
C. Webster and A. Winchester \cite{WebsterWinchester1}
gave geometric
conditions which determine whether or not every point
on the metric boundary of a graph with the standard path
metric is a Busemann point,
and an example of a graph which
admits non-Busemann points in its metric boundary.
However,
there are few examples
of metric spaces where the metric boundary 
or Busemann points are explicitly known
(e.g. \cite{KMN} and \cite{Walsh2}).

\subsubsection*{The metric boundary of Teichm\"uller space}
From the Kerckhoff's formula \eqref{eq:Kerckhoff_formula},
we often call the geometry of the Teichm\"uller distance
the \emph{extremal length geometry on Teichm\"uller space}.
In \cite{GM},
F. Gardiner and H. Masur
defined a natural boundary in the extremal length geometry.
The boundary in the extremal length geometry
is recently called the \emph{Gardiner-Masur boundary}
 (cf. \S\ref{subsec:GMbdy}).
%(cf. \cite{Mi0} and \cite{Mi1}).
In \cite{LiuSu},
L. Liu and W. Su recently observed that
the metric boundary
with respect to the Teichm\"uller distance
is canonically identified with the Gardiner-Masur boundary.
The author showed
%by applying the Gardiner-Masur compactification
%(that is, the metric boundary of the Teichm\"uller distance)
%that Teichm\"uller space is realized as 
%a ``hyperboloid" in a cone and its boundary points are rays in a ``light cone".
%Moreover, 
that the Gromov product with respect to Teichm\"uller distance
extends continuously to the boundary,
and the extension is canonically
related to the intersection number function
on the space of measured foliations
(for details, see \cite{Mi4}).
Thus, the metric compactification of Teichm\"uller space
connects between topological and analytic aspects of Teichm\"uller space.

The metric compactification
is thought of  the Martin compactification
in some potential theory.
Indeed,
M. Akian, S. Gaubert and C. Walsh \cite{AGW}
%generalized the max-plus Martin boundary
%in the idempotent analysis
%to the geometry of metric boundary
%of proper metric spaces.
correlated the metric boundary
with the Martin boundary in max-plus idempotent analysis.
%They developed a potential theory
%in which the harmonic functions $u$
In max-plus potential theory,
harmonic functions
$u$ on the metric space $(M,\rho)$ take the form
%on a proper metric space $(M,\rho)$
%are defined as
%a function on $M$ satisfying
$$
u(x)=\sup_{y\in M}\{A(x,y)+u(y)\},
$$
where $A(x,y)$ is the max-plus potential kernel
which is a function on $M\times M$ with appropriate properties
(cf. \S2 of \cite{AGW}).
%which
%give the stationary solutions of deterministic
%optimal control problems with additive reward.
In this setting,
the metric boundary coincides with
the max-plus Martin boundary
with the generating kernel $A(x,y)=-\rho(x,y)$,
and the set of Busemann points
is exactly equal to the intersection of
the metric boundary and the minimal Martin space
(cf. Corollary 7.13 in \cite{AGW},
see also the discussion after Example 7.11 in \cite{AGW}).

Any minimal point in the Martin compactification
is harmonic in max-plus potential theory
(cf. Proposition 4.4 in \cite{AGW}).
In our case,
for instance,
the log-extremal length function
of a uniquely ergodic measured foliation $G$
\begin{equation} \label{eq:horofunction_G}
T(X)\ni y \mapsto 
\frac{1}{2}\left(\log \ext_y(G)-\log \ext_{x_0}(G)\right)
\end{equation}
is harmonic in the max-plus potential theory
since the projective class $[G]$ is a Busemann point
(and hence, is a minimal Martin point)
and \eqref{eq:horofunction_G} is the corresponding horofunction,
where $x_0=(X,id)$ is the basepoint of $T(X)$
(cf. \cite{LiuSu} and \cite{Mi3}. See also \S\ref{subsec:teichmuller-space}).
Thus, it is expected that there is a new kind of potential theory closely related to
the extremal length geometry of the Teichm\"uller space
which may differ from Kaimanovich and Masur's potential theory,
and
Busemann points and non-Busemann points
will have a crucial rule in the potential theory.
(cf. \cite{Masur4} and \cite{KM2}. See also \cite{KM1}).

\subsection{Results}
Let $X$ be a Riemann surface of finite analytic type $(g,m)$ with $2g-2+m>0$.
The \emph{Teichm\"uller space} $T(X)$ of $X$ is a
quasiconformal deformation space
of marked Riemann surfaces with same type as $X$.
Teichm\"uller space $T(X)$ admits a canonical distance,
called the \emph{Teichm\"uller distance} $d_T$
(cf. \S\ref{subsec:teichmuller-space}).

The aim of this paper is to show the following.

\begin{theorem}[Non-Busemann points]
\label{thm:main}
When $3g-3+m\ge 2$,
the metric boundary of the Teichm\"uller space
with respect to the Teichm\"uller distance
contains non-Busemann points.
\end{theorem}
Namely,
the metric boundary of Teichm\"uller space
contains non-minimal points
in terms of the max-plus potential theory.

On the other hand,
when $3g-3+m=1$,
the Teichm\"uller space endowed with the Teichm\"uller distance
is isometric to the Poincar\'e hyperbolic disk.
Hence,
every point in the metric boundary is a Busemann point.
Furthermore,
in this case,
the metric boundary of Teichm\"uller space equipped with the Teichm\"uller distance
coincides with the Thurston boundary (cf. e.g. \cite{Mi0}).

It is known that any point in the metric boundary of a complete ${\rm CAT}(0)$-space
is a Busemann point (cf. Corollary II.8.20 of \cite{BriF}).
Therefore,
we conclude the following corollary,
which is first observed by Masur \cite{Masur0}.

\begin{corollary}[Masur]
\label{coro:not_CAT0}
When $3g-3+m\ge 2$,
Teichm\"uller space equipped with the Teichm\"uller distance
is not a ${\rm CAT}(0)$-space.
\end{corollary}

\subsection{The Gardiner-Masur boundary}
\label{subsec:GMbdy}
Let $\mathcal{S}$ be the set of homotopy classes of
non-trivial and non-peripheral simple closed curves
on $X$.
We denote by $\ext_y(\alpha)$ the \emph{extremal length}
of $\alpha$ for $y\in T(X)$ (cf. \S\ref{subsubsec:teichmuller-distance}).
In a beautiful paper \cite{GM},
Gardiner and Masur proved that
the mapping
$$
T(X)\ni y\mapsto
[\mathcal{S}\ni \alpha\mapsto \ext_y(\alpha)^{1/2}]
\in {\rm P}\mathbb{R}_+^\mathcal{S}
$$
is an embedding and the image is  relatively compact,
where $\mathbb{R}_+=\{x\ge 0\}$
and ${\rm P}\mathbb{R}_+^\mathcal{S}=
(\mathbb{R}_+^\mathcal{S}-\{0\})/\mathbb{R}_{>0}$.
The closure
%$\cl{T(X)}$
of the image is called the \emph{Gardiner-Masur compactification}.
The \emph{Gardiner-Masur boundary}
$\partial_{GM}T(X)$
is the complement of the image from the Gardiner-Masur compactification.
Gardiner and Masur also showed that the Gardiner-Masur boundary
contains the space $\mathcal{PMF}$ of projective measured foliations
(cf. Theorem 7.1 in \cite{GM}).

From Liu and Su's result,
Theorem \ref{thm:main}
is a consequence of the following theorem.

\begin{theorem}[Invisibility via almost geodesic rays]
\label{thm:main2}
When
$3g-3+m\ge 2$,
%Let $G$ be a maximal rational measured foliation.
% whose support contains
%at least two simple closed curves.
the projective class of a maximal rational measured foliation
cannot be the limit
of any almost geodesic ray
in the Gardiner-Masur compactification.
\end{theorem}

Related to our theorems,
%the author already observed in \cite{Mi1}
%that any Teichm\"uller geodesic ray
%does not converge to the projective class $[G]$
%when $G$ is a rational foliation whose support consists of at least two curves.
In \cite{Mi1} the author already observed that no Teichm\"uller ray
converges to the projective class of $[G]$
if $[G]$ is a rational foliation whose support consists of two or more curves.
%This observation also implies Corollary \ref{coro:not_CAT0} (cf. Corollary II.8.13 of \cite{BriF}).
%However,
%the author does not know whether Theorem \ref{thm:main2} is a consequence of this result.
%It is interesting to characterize the projective measured foliations which are not Busemann points.

In contrast to Theorem \ref{thm:main2},
Theorem 7.1 in \cite{GM} and Theorem 3 in \cite{Mi3}
assert
that
when a measured foliation $G$ is either a weighted simple closed curve
or a uniquely ergodic measured foliation,
the projective class $[G]\in \mathcal{PMF}
\subset \partial_{GM}T(X)$ is the limit of the Teichm\"uller ray
associated to $[G]$.
Therefore,
the projective classes $[G]\in \mathcal{PMF}
\subset \partial_{GM}T(X)$
of such measured foliations
are Busemann points with respect to the Teichm\"uller distance.

\subsection{Thurston's asymmetric metric vs Teichm\"uller distance}
Recently,
C. Walsh defined the horofunction boundaries for asymmetric
metric spaces,
and observed that the horofunction boundary of Teichm\"uller space
with respect to the Thurston's asymmetric Lipschitz metric
can be canonically
identified with the Thurston boundary.
In \cite{Walsh},
he also showed that every point in the Thurston boundary
is a Busemann point with respect to the Thurston's asymmetric Lipschitz metric.

Thurston's
asymmetric Lipschitz metric coincides with
the length spectrum asymmetric metric with respect to
the hyperbolic lengths of simple closed curves,
while the Teichm\"uller distance is nothing but the length spectrum metric
with respect to the extremal lengths of simple closed curves
from
Kerckhoff's formula (cf. \eqref{eq:Kerckhoff_formula}.
See also \cite{LPST}).
Since hyperbolic and extremal
lengths are fundamental conformal invariants
in the theory of Riemann surfaces,
the distances on Teichm\"uller space associated to these quantities
should be
essential in Teichm\"uller theory.
Thus, it
is natural to compare the properties of these distances and
the associated metric compactifications they induce.
Theorem \ref{thm:main} and Walsh's results above
imply that
the asymptotic geometry with respect to the Teichm\"uller distance
is more complicated than that with respect to
Thurston's asymmetric Lipschitz metric\footnote{The author should notice that
he does NOT mean here that the geometry of Thurston's asymmetric Lipschitz metric
is simpler than that of the Teichm\"uller distance.}

\subsection{Convex realization of Teichm\"uller distance}
%In \cite{KMN},
%A. Karlsson, V. Metz and G. Noskov
%observed that 
%
%
A. Papadopoulos posed a problem
\begin{quote}
Realize Teichm\"uller space
as a bounded convex set somewhere and
study the Hilbert metric on it
\end{quote}
(cf. Problem 13 in \cite{Papadopoulos}).
Following the problem,
it is natural to ask whether
$(T(X),d_T)$ can be realized as a bounded convex domain with the Hilbert metric.
In \cite{Walsh2},
C. Walsh gave a criterion that every horofunction
of the Hilbert geometry on given convex domains
is a Busemann point.
From Walsh's criterion,
if $(T(X),d_T)$ were realized as the Hilbert geometry
of a convex domain,
the convex domain seems to be complicated.
For instance,
Theorem \ref{thm:main} asserts that
the Teichm\"uller space $(T(X),d_T)$
with the Teichm\"uller distance
cannot be realized as
the Hilbert geometry on any polytope,
since all horofunctions of the Hilbert geometry on
a polytope are Busemann points
(see also \cite{KMN}).

\subsection{Plan of this paper}
This paper is organized as follows.
In \S2,
we recall the definitions and properties of some ingredients
in Teichm\"uller theory,
including the extremal length and the Teichm\"uller distance.
In \S3,
we discuss the metric boundaries of metric spaces,
and show that any almost geodesic ray in Teichm\"uller space
converges
in the Gardiner-Masur compactification.
In \cite{LiuSu},
Liu and Su also proved this convergence using properties of the metric boundary.
For the reader's convenience we give a simple proof applying Teichm\"uller theory.
%Liu and Su have proved this convergence using properties of the metric boundary,
%for the reader's convenience,
%%this convergence follows from properties of the metric boundary and
%%Liu and Su's work in \cite{LiuSu},
%we shall give a simple proof here by using the Teichm\"uller theory

We treat measured foliations whose projective classes are
the limits of almost geodesic rays
in \S4 and \S5.
Indeed,
in \S5,
we will observe that when a measured foliation whose projective class is
the limit of an almost geodesic ray
has a foliated annulus as its component,
%the geodesic realization of 
any simple closed curve
%with respect to the metric induced from
%the holomorphic quadratic differential associated to the measured foliation
is not so \emph{twisted} in the characteristic annulus corresponding to
the foliated annulus through
the almost geodesic ray (cf. Lemma \ref{lem:twisting_number}).
This is a key for getting our result.
In \S6,
we give the proof of Theorem \ref{thm:main2} by contradiction.
Indeed,
under the assumption that the projective class of a maximal measured foliation
$G$ is the limit of an almost geodesic ray,
we calculate the limit
of a given almost geodesic ray.
%(in ${\rm P}\mathbb{R}_+^\mathcal{S}$).
On the other hand,
we can check that the limit function cannot be
the intersection number function
associated to $G$.

For getting the limit function,
we will make use of Kerckhoff's calculation
of the extremal length along the Teichm\"uller ray
given in \cite{Ker}.
One of the reasons Kerckhoff's
calculations work is
that on any almost geodesic ray,
simple closed curves
satisfy a \emph{non-twisting property} along the core curves of characteristic annuli.
This property is discussed in section 5 (see also \S\ref{subsec:idea}).
%
%One of the reasons why Kerckhoff's calculation works is
%that on any almost geodesic ray,
%simple closed curves satisfy a \emph{non-twisted property}
%along the core curve of the characteristic annuli discussed in \S5
%(see also \S\ref{subsec:idea}).

\subsection*{Remark}
After submitting this paper,
C. Walsh informed the author that
he obtained a characterization of Busemann points in the horofunction boundary
of the Teichm\"uller space with respect to the Teichm\"uller distance,
and also found non-Busemann points in the boundary
(see \cite{Walsh3}).
However,
his method is different from ours.

\subsection*{Acknowledgements}
The author would like to express his heartfelt gratitude to Professor Athanase Papadopoulos
for his careful reading and many beneficial comments.
He also thanks Professor Weixu Su for his careful reading and useful comments,
and Professor C. Walsh  for valuable discussions and informing his work \cite{Walsh3}.
He thanks the referee for his/her careful reading and many valuable and fruitful
comments.

\subsection*{Notation}
For two functions $f(t)$ and $g(t)$ with variable $t$,
$f(t)\asymp g(t)$ means that $f(t)$ and $g(t)$ are comparable
in the sense that there are positive numbers $B_1$ and $B_2$
independent of the parameter $t$ such that $B_1 g(t)\le f(t)\le B_2g(t)$.

\section{Extremal length and Teichm\"uller theory}

\subsection{Extremal length}
For an introduction to the theory of extremal length,
See Ahlfors' books \cite{Ahlfors} and \cite{Ahlfors2}.

\subsubsection{Extremal length of a family of rectifiable curves}
Let $\Gamma$ be a family of rectifiable curves on a Riemann surface $R$.
The \emph{extremal length} of $\Gamma$ (on $R$)
is defined by
\begin{equation} \label{eq:definition_extremal_length_analytic}
\ext(\Gamma)=\sup_\rho \frac{L_\rho(\Gamma)^2}{A(\rho)}
\end{equation}
where supremum runs over all measurable conformal metrics
$\rho=\rho(z)|dz|^2$ and
$$
L_\rho(\Gamma)
=
\inf_{\gamma\in \Gamma}\int_\gamma \rho(z)^{1/2}|dz| \quad
\mbox{and}\quad
A(\rho)
=
\iint_R\rho(z)dxdy.
$$
When a metric $\rho$ attains the supremum in \eqref{eq:definition_extremal_length_analytic},
it is called an \emph{extremal metric}.
Extremal length is a \emph{conformal invariant}
and a $K$-quasiconformal $K$-invariant in the sense that
\begin{equation}
\label{eq:conformal_invariant}
\ext (h(\Gamma))\le K\,\ext (\Gamma)
\end{equation}
for a $K$-quasiconformal mapping $h:R\to h(R)$,
a Riemann surface $R$,
and a family $\Gamma$ of rectifiable curves on $R$.

\begin{proposition}[See \cite{Ahlfors} and \cite{Ahlfors2}]
\label{prop:extremal_length}
Let $\Gamma_1$ and $\Gamma_2$ be two families of rectifiable curves
on a Riemann surface $R$.
\begin{itemize}
\item[{\rm (1)}]
If any curve in $\Gamma_1$ is contained in a subdomain $D_1$ of $R$,
the extremal length of $\Gamma_1$ on $R$ is equal to
the extremal length of $\Gamma_1$ on $D_1$.
\item[{\rm (2)}]
If any curve in $\Gamma_2$ contains a curve in $\Gamma_1$,
$\ext (\Gamma_1)\le \ext (\Gamma_2)$.
\item[{\rm (3)}]
Let $\Gamma_3$ be a family of closed curves,
and suppose that $\Gamma_1$ and $\Gamma_2$ are contained in 
disjoint open sets in $R$.
If every curve in $\Gamma_3$ contains
a curve in $\Gamma_1$ and a curve in $\Gamma_2$,
$\ext(\Gamma_3)\ge \ext(\Gamma_1)+\ext(\Gamma_2)$.
\end{itemize}
\end{proposition}

\subsubsection{Extremal length and modulus of annulus}
For an annulus $A$,
we denote by $\ext (A)$ the extremal length
of the family of simple closed curves which are homotopic to the core
curve of $A$.
The \emph{modulus} of $A$ is the reciprocal of the extremal length of $A$.
If $A$ is conformally equivalent to the flat annulus $\{r_1<|z|<r_2\}$,
one can see that ${\rm Mod}(A)=\frac{1}{2\pi}\log (r_2/r_1)$.

\begin{proposition}[cf. Proposition 9.1 of \cite{Mi1}]
\label{prop:extremal_length_upper}
Let $A$ be an annulus.
Let $\{\beta_k\}_{k=1}^N$ be mutually disjoint Jordan arcs joining components
of $\partial A$ such that $\beta_{k-1}$ and $\beta_{k+1}$
divides $\beta_k$ from the
other arcs {\rm (}set $\beta_{N+1}=\beta_1${\rm )}.
Let $\Gamma_k$ be the set of paths in $A-\cup_{k=1}^N\beta_k$
connecting $\beta_k$ and $\beta_{k+1}$.
Let $\rho$ be the extremal metric for $\ext (A)$ on $A$ such that $A(\rho)=1$.
Suppose that the $\rho$-length of $\beta_k$ is bounded for all $k=1,\cdots,N$.
Then,
$$
\ext (A)^{1/2}\le
\left(
\sum_{k=1}^N\ext (\Gamma_k)
\right)^{1/2}
+B
$$
where $B$ is the sum  of $\rho$-lengths of $\beta_k$'s.
\end{proposition}

\subsubsection{Extremal lengths of simple closed curves}
For a Riemann surface $Y$ and a simple closed curve $\beta$ on $Y$,
we define the \emph{extremal length} $\ext_Y(\beta)$ of $\beta$ on $Y$
is the extremal length of a family of rectifiable closed curves on $Y$
homotopic to $\beta$.
The extremal length is represented geometrically by
\begin{equation} \label{eq:definition_extremal_length_geometric}
\ext_Y(\beta)=1/\sup_A \{{\rm Mod}(A)\}=\inf_A\ext (A)
\end{equation}
where $A$ runs all annuli on $Y$ whose core is homotopic to $\beta$
(cf. e.g. \cite{Ker} and \cite{Strebel}).

\subsection{Measured foliations}
\label{subsec:MF}
%Let $\mathcal{S}$ be the set of homotopy classes of
%non-trivial and non-peripheral simple closed curves
%on $X$.
The set of formal products $\mathbb{R}_+\otimes \mathcal{S}
=\{t\alpha\mid t\ge 0,\alpha\in \mathcal{S}\}$
is embedded into $\mathbb{R}_+^{\mathcal{S}}$
via the (geometric) intersection number:
$$
\mathbb{R}_+\otimes \mathcal{S}\ni t\alpha\mapsto
[\mathcal{S}\ni \beta\mapsto t\,i(\alpha,\beta)]\in \mathbb{R}_+^{\mathcal{S}}.
$$
The closure $\mathcal{MF}=\mathcal{MF}(X)$
of the image in $\mathbb{R}_+^{\mathcal{S}}$ is called the
\emph{space of measured foliations}
on $X$,
where we topologize $\mathbb{R}_+^{\mathcal{S}}$ with
the topology of the pointwise convergence.
The \emph{space $\mathcal{PMF}=\mathcal{PMF}(X)$
of projective measured foliations} is the quotient space
$(\mathcal{MF}-\{0\})/\mathbb{R}_{>0}$.
It is known that $\mathcal{MF}$ and $\mathcal{PMF}$
are homeomorphic to $\mathbb{R}^{6g-6+2n}$
and $S^{6g-7+2n}$,
respectively (cf. \cite{FLP}).
It is also known that when we define
the intersection number between weighted simple closed curves
$t\alpha,s\beta\in \mathbb{R}_+\otimes \mathcal{S}$
by the homogeneous equation
$i(t\alpha,s\beta)=ts\,i(\alpha,\beta)$,
the intersection number function extends continuously
on $\mathcal{MF}\times \mathcal{MF}$.
To a measured foliation $G$,
we associate
a singular foliation and a transverse measure to the underlying foliation (cf. \cite{FLP}).
In this paper,
we denote by $\int_\beta G$
the integration of the corresponding transverse measure over a path $\beta$
transverse to the underlying foliation.
%A measured lamination is represented by
%a family of $C^1$-functions $v=\{v_i\}_i$ on
%the outside of a discrete set $E_0$ on $X$,
%where $v_i$ is defined on an open set $U_i$ on $X-E_0$
%and $dv_i=\pm dv_j$ when $U_i\cap U_j\ne \emptyset$
%(cf. \cite{GM}).
%By definition,
%$$
%i(\beta,G)=\inf_{\beta'\in \beta}\int_{\beta'}G=\inf_{\beta'\in \beta}\int_{\beta'}|dv|.
%$$

A measured foliation $G$ is called \emph{rational}
if $G$ satisfies
$$
i(\beta,G)=\sum_{i=1}^k w_ii(\beta,\alpha_i)
$$
for some $w_i>0$ and $\alpha_i\in \mathcal{S}$
such that $i(\alpha_i,\alpha_j)=0$ and $\alpha_i\ne \alpha_j$
for $i,j=1,\cdots,k$ with $i\ne j$.
We write $G=\sum_{i=1}^k w_k\alpha_k$ for such a measured foliation.
A rational measured foliation $G=\sum_{i=1}^k w_k\alpha_k$ is
\emph{maximal} if any component of $X-\cup_{i=1}^k \alpha_i$ is a pair of pants.
In this case,
$k=3g-3+m$.

In \cite{Ker},
S. Kerckhoff showed that when we put $\ext_X(t\beta)=t^2\ext_X(\beta)$
for $t\beta\in \mathbb{R}_+\otimes \mathcal{S}$,
the extremal length extends continuously to $\mathcal{MF}$.
We define
the \emph{unit sphere}
$$
\mathcal{MF}_1=\{F\in \mathcal{MF}\mid \ext_X(F)=1\},
$$
in $\mathcal{MF}$
which is homeomorphic to $\mathcal{PMF}$ via the projection
$\mathcal{MF}-\{0\}\to\mathcal{PMF}$.

It is  known
that
the following inequality,
called \emph{Minsky's inequality}
holds:
\begin{equation}
\label{eq:minsky_inequality}
i(F,G)^2\le \ext_X(F)\,\ext_X(G)
\end{equation}
for all $F,G\in \mathcal{MF}$
(cf. Lemma 5.1 of \cite{Minsky2}).
Minsky's inequality
% \eqref{eq:minsky_inequality}
is sharp in the sense that for any $G\in \mathcal{MF}-\{0\}$,
there is a unique $F\in \mathcal{MF}$,
up to multiplying a positive constant,
which satisfies the equality in
\eqref{eq:minsky_inequality}
(cf. Theorem 5.1 in \cite{GM}).
%
%\begin{lemma}[Rough additivity]
%\label{lem:rough_additivity}
%Let $F=F_1+\cdots+F_k\in \mathcal{MF}$.
%Then,
%it holds
%$$
%(3g-3+m)^{-1}\ext_y(F)
%\le
%\sum_{i=1}^k\ext_y(F_i)
%\le \ext_y(F)
%$$
%for all $y\in T(X)$.
%\end{lemma}
%
%\begin{proof}
%Since the supports of $F_i$ are mutually disjoint,
%we may assume that each $F_i$ is rational from the continuity of the extremal length on $\mathcal{MF}$.
%Let $F_i=m_i\alpha_i$.
%Then,
%by the Schwarz inequality,
%we have
%\begin{align*}
%\ext_y(F)
%&=\|J_{F,y}\|=\sum_{i=1}^k m_i\ell_{J_{F,y}}(\alpha_i) 
%=\left(\sum_{i=1}^k 1\cdot (m_i\ell_{J_{F,y}}(\alpha_i))\right)^2/\|J_{F,y}\| \\
%&\le k \cdot
%\sum_{i=1}^k m_i^2\frac{\ell_{J_{F,y}}(\alpha_i)^2}{\|J_{F,y}\|}
%\le k\sum_{i=1}^km_i^2\ext_y(\alpha_i)=
% k\sum_{i=1}^k\ext_y(F_i),
%\end{align*}
%which implies the left-hand side of the inequality which we now claim.
%
%Since the modulus of the characteristic annulus of $J_{F,y}$ with respect to $\alpha_i$ is equal to $m_i/\ell_{J_{F,y}}(\alpha_i)$,
%$\ext_(\alpha_i)\le \ell_{J_{F,y}}(\alpha_i)/m_i$ from the geometric definition
%of the extremal length.
%Therefore,
%we deduce
%\begin{align*}
%\sum_{i=1}^k\ext_y(F_i)
%&=\sum_{i=1}^km_i^2 \ext_y(\alpha_i)
%\le
%\sum_{i=1}^km_i^2 \ell_{J_{F,y}}(\alpha_i)/m_i \\
%&=
%\sum_{i=1}^km_i \ell_{J_{F,y}}(\alpha_i)
%=\ext_y(F),
%\end{align*}
%which is what we wanted.
%\end{proof}
%

\subsection{Teichm\"uller space}
\label{subsec:teichmuller-space}
The \emph{Teichm\"uller space} $T(X)$ of $X$
is the set of equivalence classes of marked Riemann surfaces
$(Y,f)$ where $Y$ is a Riemann surface and $f:X\to Y$
is a quasiconformal mapping.
Two marked Riemann surfaces $(Y_1,f_1)$ and $(Y_2,f_2)$ are \emph{Teichm\"uller equivalent}
if there is a conformal mapping $h:Y_1\to Y_2$ which is homotopic to $f_2\circ f_1^{-1}$.
Throughout this paper,
we consider the Teichm\"uller space as a pointed space with
basepoint $x_0=(X,id)$.

%For $F\in \mathcal{MF}$ and $y=(Y,f)\in T(X)$,
%we define
%$$
%\ext_y(F)=\ext_Y(f(F)).
%$$
%From the distortion property of extremal length (cf. \cite{Ahlfors}),
%we can see that the function
%$$
%T(X)\times \mathcal{MF}\ni (y,F)\mapsto \ext_y(F)
%$$
%is continuous.

\subsubsection{Teichm\"uller distance and Kerckhoff's formula}
\label{subsubsec:teichmuller-distance}
The \emph{Teichm\"uller distance} between $y_1=(Y_1,f_1)$ and
$y_2=(Y_2,f_2)\in T(X)$ is,
by definition,
the half of the logarithm
of the maximal dilatation of
the extremal quasiconformal mapping
between $Y_1$ and $Y_2$ preserving their markings
(cf. \cite{IT})

For $F\in \mathcal{MF}$ and $y=(Y,f)\in T(X)$,
the \emph{extremal length} of $F$ on $y$ 
is defined by
$$
\ext_y(F)=\ext_Y(f(F)).
$$
In \cite{Ker},
S. Kerckhoff gave the geometric interpretation of the Teichm\"uller distance
in terms of the extremal lengths of measured foliations:
\begin{equation} \label{eq:Kerckhoff_formula}
d_T(y_1,y_2)=
\frac{1}{2}
\log 
\sup_{H\in \mathcal{MF}-\{0\}}
\frac{\ext_{y_1}(H)}{\ext_{y_2}(H)}
=
\frac{1}{2}
\log 
\max_{H\in \mathcal{MF}_1}
\frac{\ext_{y_1}(H)}{\ext_{y_2}(H)}.
\end{equation}
Teichm\"uller space is topologized with the Teichm\"uller distance.
Under this topology,
the extremal length of a measured foliation varies continuously on
$T(X)$ (See also \eqref{eq:conformal_invariant}).

\subsubsection{Quadratic differentials and Hubbard-Masur-Gardiner's theorem}
For a holomorphic quadratic differential $q=q(z)dz^2$ on a Riemann surface $Y$,
we define a singular flat metric $|q|=|q(z)||dz|^2$.
We call here this metric the \emph{$q$-metric}.

In \cite{HM},
Hubbard and Masur observed that
for $y=(Y,f)\in T(X)$ and $G\in \mathcal{MF}-\{0\}$,
there is a unique holomorphic quadratic differential $J_{G,y}$ on $Y$
whose vertical foliation is equal to $f(G)$
when $X$ is closed.
In \cite{Gar},
Gardiner extends their result to punctured surfaces
by applying his minimal Dirichlet principle for measured foliations.
In any case,
we obtain 
$$
i(\beta,G)=\inf_{\beta'\in f(\beta)}\int_{\beta'}
\left|
{\rm Re}\sqrt{J_{G,y}}
\right|
$$
for all $\beta\in \mathcal{S}$,
and from
the minimum Dirichlet principle
%Namely,
%$\left|
%{\rm Re}\sqrt{J_{G,y}}
%\right|=f(G)$ as measured foliations.
%We can see that
$$
\ext_{y}(G)=\|J_{G,y}\|=\iint_Y|J_{G,y}|.
$$
Namely,
the extremal length is the area of the $J_{G,y}$-metric.
From the uniquness of the differential,
we can see that $J_{tG,y}=t^2J_{G,y}$ for $t>0$ and $G\in \mathcal{MF}$.
When $G$ is rational,
we call the differential $J_{G,y}$ the \emph{Jenkins-Strebel
differential for $G$}.

%The following is proved in A.Marden and K.Strebel
%(see Theorem 3.2 of \cite{MS3}. See also \cite{Gar}).
%
%\begin{proposition}[Minimal norm property]
%\label{prop:minimal_norm_property}
%Let $G$ and $H$ be measured foliations 
%and $y=(Y,f)\in T(X)$.
%Suppose that
%$f(H)$ is represented by a system of $C^1$-functions $v=\{v_i\}_i$ on $Y$.
%If $i(\alpha,G)\le i(\alpha,H)$ for all $\alpha\in \mathcal{S}$,
%then
%$$
%\|J_{G,y}\|\le \int_Y(v_\xi^2+v_\eta^2)d\xi d\eta
%$$
%whete $\xi+i\eta$ is a complex local coordinate on $Y$.
%The equality hold if and only if $|dv_i|=|{\rm Re}\sqrt{J_{G,y}}|$ on the domain 
%of definition of $v_i$ for all $i$.
%\end{proposition}

\subsection{Teichm\"uller rays}
\label{subsec:Teichmuller_rays}
Let $x=(X,f)\in T(X)$
and $[G]\in \mathcal{PMF}$.
Let $R_{G,x_0}(t)$ be the point of $T(X)$
represented by the the Beltrami coefficient
\begin{equation} \label{eq:beltrami}
\tanh(t)\frac{|J_{G,x_0}|}{J_{G,x_0}}
\end{equation}
for $t\ge 0$.
Notice that the Beltrami
differential \eqref{eq:beltrami} depends only on the projective class of $G$.
Teichm\"uller's theorem asserts that
$$
[0,\infty)\ni t\mapsto R_{G,x_0}(t)\in T(X)
$$
is
an isometric embedding with respect to the Teichm\"uller distance
(cf. \cite{IT}).
We call $R_{G,x_0}$ the \emph{Teichm\"uller (geodesic) ray
associated to $[G]\in \mathcal{PMF}$}.
%Combining Teichm\"uller's theorem
%and the height theorem for holomorphic
%quadratic differentials (cf. \cite{Strebel}),
It is known that
$$
\mathcal{PMF}\times [0,\infty)/(\mathcal{PMF}\times \{0\})\ni ([G],t)
\mapsto R_{G,x_0}(t)\in T(X)
$$
is a homeomorphism (cf. \cite{Bers2} and \cite{IT}).
One can see that
\begin{equation}
\label{eq:extremal_length_teichmuller_rays}
\ext_{R_{G,x_0}(t)}(G)=e^{-2t}\ext_{x_0}(G)
\end{equation}
for $G\in \mathcal{MF}$.

%\subsection{Functions $\mathcal{E}_p$}
%In \cite{Ker},
%S. Kerckhoff gave an explicit formula
%\begin{equation}
%\label{eq:limit_G}
%\mathcal{E}_{\mathcal{G}_x([G])}(F)
%=
%\left\{
%\sum_{i=1}^k m_i\,i(F,\alpha_i)^2
%\right\}^{1/2}
%\end{equation}
%where
%\begin{equation} \label{eq:definition_G}
%G=\sum_{i=1}^kw_i\alpha_i
%\end{equation}
%with $\alpha_i\in \mathcal{S}$
%and $w_i>0$,
%and $m_i$ is the modulus of the characteristic annulus of
%the holomorphic quadratic
%differential for $G$ whose core is homotopic to $\alpha_i$
%(see also Appendix of \cite{Mi1}).

\subsection{Gardiner-Masur boundary revisited}
\label{subsec:GM_revisited}
For $y=(Y,f)\in T(X)$,
we let
$$
K_y=e^{2d_T(x_0,y)}.
$$
Namely,
$K_y$ is the maximal dilatation of the extremal quasiconformal mapping
between $X$ to $Y$ homotopic to the marking $f$.
Consider a continuous function on $\mathcal{MF}$
defined by
\begin{equation}
\label{eq:function_E}
\mathcal{E}_y(F)=\left(
\frac{\ext_y(F)}{K_y}\right)^{1/2}
\end{equation}
for $y\in T(X)$.
Then,
in \cite{Mi1},
the author observed that for any $p\in \partial_{GM}T(X)$,
there is a function $\mathcal{E}_p$ on $\mathcal{MF}$
such that the function $\mathcal{S}\ni \beta\mapsto \mathcal{E}_p(\beta)$
represents $p$ and when a sequence $\{y_n\}_{n=1}^\infty\subset T(X)$
converges to $p$ in the Gardiner-Masur compactification,
there are $t_0>0$ and a subsequence $\{y_{n_j}\}_{j=1}^\infty$ of $\{y_n\}_{n=1}^\infty$
such that $\mathcal{E}_{y_{n_j}}$ converges to $t_0\,\mathcal{E}_p$
uniformly on any compact set of $\mathcal{MF}$.
%(cf. Theorem 1.1 in \cite{Mi1}).

As noticed in \S\ref{subsec:GMbdy},
the space $\mathcal{PMF}$ of projective measured foliations
is contained in $\partial_{GM}T(X)$.
By definition,
for $[G]\in \mathcal{PMF}$,
the function $\mathcal{E}_{[G]}$ corresponding to $[G]$
is nothing but
a positive multiple of the intersection number function
associated to $G$.
Namely,
there is a constant $t_0=t_0([G],x_0)>0$ such that
$$
\mathcal{E}_{[G]}(F)=t_0\cdot i(F,G)
$$
for $F\in \mathcal{MF}$.

\section{Metric boundary and horofunction boundary}

\subsection{Metric boundary and horofunction boundary}
\label{subsec:metricboundary}
Let $(M,\rho)$ be a locally compact metric space.
Let $C(M)$ be the space of complex valued continuous functions on $M$,
equipped with the topology of uniform convergence on compact
subsets of $M$.
Let $C_*(M)$ be $C(M)$ factored by
the addition of constant functions.
%where every constant functions acts on $C(M)$ by translation.
%For $y\in M$ we set $\psi_y(x)=\rho(x,y)$.
Then,
the mapping
$$
M\ni y\mapsto [M\ni x\mapsto \rho(x ,y)]\in C(M)
$$
is a continuous embedding.
Furthermore,
this embedding descends to a continuous embedding
from $M$ into $C_*(M)$.
The closure $\mathcal{C}\ell(M)\subset C_*(M)$
of the image of this embedding is called
the \emph{horofunction compactification}
and the complement $\mathcal{C}\ell(M)\setminus M$
is said to be the \emph{horofunction boundary} of $M$
(cf. \cite{Gromov},
\cite{BriF},
and \cite{Rieffel}).
M. Rieffel pointed out that the horofunction boundary of $M$ is canonically identified with
the metric boundary of $M$
as discussed in the introduction
(cf. \S4 in \cite{Rieffel}).

In \cite{LiuSu},
L. Liu and W. Su showed
that the horofunction compactification of the Teichm\"uller space
with the Teichm\"uller distance is canonically identified with
the Gardiner-Masur compactification
in the sense that the identity mapping $T(X)\to T(X)$ extends
to a homeomorphism between them.

\subsection{Almost geodesics ray}
\label{subsec:almost_geodesics}
Let $(M,\rho)$ be a metric space.
Let $T\subset [0,\infty)$ be an unbounded set with $0\in T$.
A mapping $\gamma:T\to M$ is said to be an \emph{almost geodesic ray}
if for any $\epsilon>0$ there is an $N>0$ such that
for all $t,s\in T$ with $t\ge s\ge N$,
$$
|\rho(\gamma(t),\gamma(s))+\rho(\gamma(s),\gamma(0))-t|<\epsilon
$$
(cf. Definition 4.3 of \cite{Rieffel}).
By definition,
any geodesic ray is an almost geodesic ray.
When $(M,\rho)$ is a pointed metric space,
we assume in addition that $\gamma (0)$ is equal to the basepoint
(cf. the assumption of Lemma 4.5 in \cite{Rieffel}).
By definition,
for any unbounded subset $T_0\subset T$ with $0\in T_0$,
the restriction $\gamma\mid_{T_0}:T_0\to M$ is also an almost geodesic ray.
We call the restriction a \emph{subray} of an almost geodesic ray $\gamma:T\to M$.

As noticed in the introduction,
M.Rieffel showed that any almost geodesic ray has a limit in the metric compactification.
A point of the metric boundary or the horofunction boundary
of $M$ is said to be a \emph{Busemann point}
if it is the limit point of an almost geodesic ray
(cf. \cite{Rieffel}).
%(cf. Definition 4.8 of \cite{Rieffel}).

%
%A \emph{weakly-geodesic} is a map $\gamma:T\to M$ such that
%for any $y\in M$ and $\epsilon>0$,
%there is an $N>0$ such that if $s,t\ge N$,
%$$
%|\rho(\gamma(t),\gamma(0))-t|<\epsilon
%$$
%and
%$$
%|\rho(\gamma(t),y)-\rho(\gamma(s),y)-(t-s)|<\epsilon.
%$$
%In \cite{Rieffel},
%M. Rieffel showed that any weakly-geodesic is an almost geodesic.

\subsection{Convergence of almost geodesics rays}
In this section,
we shall check that any almost geodesic ray in $T(X)$
converges in the Gardiner-Masur compactification.
%({\color{red}It can also hold for weak geodesics?}).
Although this follows from a fundamental property of
the metric boundary discussed in the previous section and Liu and Su's
work \cite{LiuSu},
we give a simple proof from Teichm\"uller theory which is of independent interest.
We remark that in \cite{Mi3}
using a different idea the author observed that
any Teichm\"uller ray $R_{G,x}(t)$ admits a limit for all $[G]$ in $\mathcal{PMF}$.
%we try giving a simple proof from Teichm\"uller theory
%and it seems to be intriguing in itself.
%Notice that
%the author observed in \cite{Mi3} that
%any Teichm\"uller ray $R_{G,x}(t)$
%%$$
%%\mathcal{G}_x([G])=\lim_{t\to \infty}R_{G,x}(t)
%%$$
%admits the limit for all $[G]\in \mathcal{PMF}$
%by using a different idea.

Let $\gamma:T\to T(X)$ be an almost geodesic ray
with basepoint $x_0\in T(X)$.
By definition,
%$\gamma$ satisfies that
$\gamma(0)=x_0$ and
for any $\epsilon>0$,
there is an $N$ such that
\begin{equation}
\label{eq:almost_geodesic}
|d_T(\gamma(t),\gamma(s))+d_T(\gamma(s),\gamma(0))-t|<\epsilon
\end{equation}
for all $t\ge s\ge N$.
From Kerckhoff's formula \eqref{eq:Kerckhoff_formula},
\eqref{eq:almost_geodesic} is equivalent to
\begin{equation}
\label{eq:almost_geodesic2}
e^{t-\epsilon}\le 
\max_{H\in \mathcal{MF}_1}
\frac{\ext_{\gamma(t)}(H)^{1/2}}{\ext_{\gamma(s)}(H)^{1/2}}\cdot
K_{\gamma(s)}^{1/2}
\le
e^{t+\epsilon}.
\end{equation}
In particular,
we have
\begin{equation}\label{eq:Kgamma-t=s}
e^{t-\epsilon}\le 
K_{\gamma(t)}^{1/2}
%=e^{d_T(x_0,\gamma(t))}
\le
e^{t+\epsilon}
\end{equation}
when we set $s=t$ in \eqref{eq:almost_geodesic2}.
Therefore,
we deduce
$$
\frac{\ext_{\gamma(t)}(H)^{1/2}}{\ext_{\gamma(s)}(H)^{1/2}}
\cdot
K_{\gamma(s)}^{1/2}
\le
e^{\epsilon}\cdot e^t
\le
e^{\epsilon}\cdot 
e^{\epsilon}K_{\gamma(t)}^{1/2},
$$
and hence
\begin{equation}
\label{eq:almost_geodesic3}
\mathcal{E}_{\gamma(t)}(H)\le e^{2\epsilon}\mathcal{E}_{\gamma(s)}(H)
\end{equation}
for all $H\in \mathcal{MF}$
and $t\ge s\ge N$
(cf. \eqref{eq:function_E}).

We set
$$
\mathcal{E}'(F)=\liminf_{T\ni t\to \infty}\mathcal{E}_{\gamma(t)}(F)
$$
for $F\in \mathcal{MF}$.
%Since $\{\mathcal{E}_{\gamma(t)}\}_{t\in T}$ is normal (cf. \cite{Mi1}),
From \eqref{eq:almost_geodesic3},
for all $\beta\in \mathcal{S}$,
the limit of any converging
subray in $\{\mathcal{E}_{\gamma(t)}(\beta)\}_{t\in T}$
coincides with $\mathcal{E}'(\beta)$,
which implies that
$\gamma:T\to T(X)$ converges
in the Gardiner-Masur compactification
as $t\to \infty$.

\section{Measured foliations as Busemann points}
%Let $[G]\in \mathcal{PMF}$
%%be the projective class of a rational foliation
\subsection{Function $\mathcal{E}_{[G]}$ for $[G]\in \mathcal{PMF}$}
We first notice the following.
%Let $G\in \mathcal{MF}_1$.
%Take a sequence $\{y_n\}_n$ which converges to the projective class $[G]$
%in the Gardiner-Masur compactification.
%This means that
%there is a $t_0>0$ such that
%$\mathcal{E}_{y_n}$ converges to
%the function $\mathcal{E}_{[G]}(\,\cdot\,)=t_0\, i(\,\cdot\,,G)$
%uniformly on any compact sets of $\mathcal{MF}$
%(cf. \S\ref{subsec:GM_revisited}).

\begin{lemma}
\label{lem:t_0_is_one}
Let $G\in \mathcal{MF}_1$.
Take a sequence $\{y_n\}_n\subset T(X)$ converging to the projective class $[G]$
in the Gardiner-Masur compactification.
Then,
%there is a $t_0>0$ such that
$\mathcal{E}_{y_n}$ converges to
the intersection number function $\mathcal{MF}\ni F\mapsto i(F,G)$
of $G$
uniformly on any compact sets of $\mathcal{MF}$.
%In particular,
%$\mathcal{E}_{\gamma(t)}(\,\cdot\,)$ converges to $i(\,\cdot\,,G)$
%as $t\to \infty$
%uniformly on any compact sets of $\mathcal{MF}$.
\end{lemma}

\begin{proof}
The assumption means that
there are a subsequence $\{y_{n_j}\}_j$ and $t_0>0$ such that
$\mathcal{E}_{y_{n_j}}$ converges to
the function  $\mathcal{MF}\ni F\mapsto t_0\, i(F,G)$
uniformly on any compact sets of $\mathcal{MF}$
(cf. \S\ref{subsec:GM_revisited}).

We claim that $t_0=1$,
which means that
the limit is independent of the choice of subsequences.
%We may assume that $\mathcal{E}_{y_n}$ converges to $t_0\,i(\,\cdot\,,G)$
%uniformly on any compact set of $\mathcal{MF}$.
Indeed,
%it is well-known that
%for every $n\ge 1$,
%there is an $H_n\in \mathcal{MF}_1$ such that
%$\ext_{y_n}(H_n)=K_{y_n}\ext_{x_0}(H_n)=K_{y_n}$
%(cf. e.g. \cite{GM}). 
since $\mathcal{MF}_1$ is compact,
the convergence $\mathcal{E}_{y_n}(\,\cdot\,)\to t_0\,i(\,\cdot\,,G)$
is uniform on $\mathcal{MF}_1$.
Therefore,
$$
1=\max_{H\in \mathcal{MF}_1}
\mathcal{E}_{y_n}(H)
=\max_{H\in \mathcal{MF}_1}t_0\,i(H,G)
=t_0\ext_{x_0}(G)^{1/2}=t_0
$$
since Minsky's inequality is sharp as noticed in \S\ref{subsec:MF}.
\end{proof}

\subsection{The case where $[G]$ is a Busemann point}
\label{subsec:case_where_G_busemann}
Suppose that the projective class $[G]$ of $G\in \mathcal{MF}$ is a Busemann point
in the horofunction compactification
of Teichm\"uller space with respect to the Teichm\"uller metric.
By definition and Liu and Su's work \cite{LiuSu},
there is an almost-geodesic $\gamma:T\to T(X)$ such that
$\gamma(t)\to [G]$ in the Gardiner-Masur closure.
%This means that for any $\epsilon>0$
%there is an $N>0$ such that $t,s\in T$ with $t\ge s\ge N$,
%\begin{equation} \label{eq:almost-geodesic_Teichmuller1}
%|d_T(\gamma(t),\gamma(s))+d_T(\gamma(s),\gamma(0))-t|<\epsilon.
%\end{equation}
%%By modifying slightly,
%%we may assume that $d_T(\gamma(t),x_0)=t$
%%for all $t\in T$.
%In particular,
%\begin{equation}
%\label{eq:almost-geodesic_Teichmuller1-1}
%|d_T(\gamma(t),\gamma(0))-t|<\epsilon
%\end{equation}
%for all $t\in T$.
%From Kerckhoff's formula,
%\eqref{eq:almost-geodesic_Teichmuller1} is equivalent to
%\begin{equation}
%\label{eq:almost-geodesic_Teichmuller2}
%e^{t-\epsilon}
%\le
%\max_{H\in \mathcal{MF}_1}\frac{\ext_{\gamma(t)}(H)^{1/2}}{\ext_{\gamma(s)}(H)^{1/2}}
%\cdot
%\max_{H\in \mathcal{MF}_1}
%\ext_{\gamma(s)}(H)^{1/2}
%\le
%e^{t+\epsilon}.
%\end{equation}
%for all $t\ge s \ge N$
%since $\gamma(0)=x_0$.
%
%This means that
%there is a $t_0>0$ such that
%$\mathcal{E}_{\gamma(t)}$ converges to
%the function $\mathcal{MF}\ni F\mapsto t_0\, i(F,G)$
%uniformly on any compact sets of $\mathcal{MF}$
%(cf. \S\ref{subsec:GM_revisited}).
%We take $G_t\in \mathcal{MF}_1$  with
%$R_{G_t,x_0}(d_T(x_0,\gamma(t)))=\gamma(t)$.
%Since $d_T(\gamma(t),\gamma(0))>t-\epsilon$,
%we have
%$$
%\ext_{\gamma(t)}(G_t)\le e^\epsilon e^{-t}
%$$
%for all $t\in T$.

%\begin{lemma}
%\label{lem:t_0_is_one}
%Under the notation above,
%we have $t_0=1$.
%%In particular,
%%$\mathcal{E}_{\gamma(t)}(\,\cdot\,)$ converges to $i(\,\cdot\,,G)$
%%as $t\to \infty$
%%uniformly on any compact sets of $\mathcal{MF}$.
%\end{lemma}

\begin{lemma}[Behavior of extremal length]
\label{lem:behavior_of_K}
When an almost geodesic ray $\gamma\colon T\to T(X)$ with $\gamma(0)=x_0$
converges to the projective class
$[G]$ of $G\in \mathcal{MF}_1$,
we have
\begin{equation}
\label{eq:extremal_length_G_s}
\lim_{t\to \infty}\|J_{G,\gamma(t)}\|\cdot K_{\gamma(t)}
=
\lim_{t\to \infty}\ext_{\gamma(t)}(G)\cdot K_{\gamma(t)}=1.
\end{equation}
\end{lemma}

\begin{proof}
%Let $G_\infty$ be an accumulation point of $\{G_t\}_{t\in T}$.
%By taking a subray if necessary,
%we may assume that $G_t$ converges to $G_\infty$.
%Let $\alpha\in \mathcal{S}$.
%From 
%%the assumption,
%%and
%%the continuity of the intersection number,
%\eqref{eq:extremal_length_teichmuller_rays}
%and \eqref{eq:Kgamma-t=s},
%we have
%\begin{align}
%i(\alpha,G_\infty)
%&=
%\lim_{T\ni t\to \infty}i(\alpha,G_t)
%\le
%\lim_{T\ni t\to \infty}
%\ext_{\gamma(t)}(\alpha)^{1/2}
%\ext_{\gamma(t)}(G_t)^{1/2}
%\nonumber\\
%&\le
%e^{\epsilon}\lim_{T\ni t\to \infty}
%e^{-t}\ext_{\gamma(t)}(\alpha)^{1/2}
%\le 
%e^{2\epsilon}\lim_{T\ni t\to \infty}
%\mathcal{E}_{\gamma(t)}(\alpha)
%\nonumber\\
%&=e^{2\epsilon}\,t_0\,i(\alpha,G).
%\label{eq:comparizon_G_infty_to_G2}
%\end{align}
%Since $\epsilon>0$ is taken arbitrary,
%we get
%\begin{equation}
%\label {eq:comparizon_G_infty_to_G}
%i(\alpha,G_\infty)\le i(\alpha, t_0G)
%\end{equation}
%for all $\alpha\in \mathcal{S}$.
%%Since \eqref{eq:comparizon_G_infty_to_G2}
%%holds for all $\alpha\in \mathcal{S}$
%%and $\epsilon>0$,
%Thus,
%it follows from the Marden-Strebel's minimal norm property
%%(cf. Proposition \ref{prop:minimal_norm_property})
%that
%\begin{equation}
%\label{eq:t_0_is_greater_than_one}
%1=\ext_{x_0}(G_\infty)=\|J_{G_\infty,x_0}\|\le \|J_{t_0G,x_0}\|
%=t_0^2\|J_{G,x_0}\|
%=t_0^2,
%\end{equation}
%and hence $t_0\ge 1$
%(see Theorem 3.2 of \cite{MS3}. See also \cite{Gar}).

From \eqref{eq:Kgamma-t=s} and Lemma \ref{lem:t_0_is_one},
by dividing every term in \eqref{eq:almost_geodesic2}
by $K_{\gamma(t)}^{1/2}=e^{d_T(x_0,\gamma(t))}$
and letting $t\to \infty$,
we get
\begin{equation}
\label{eq:almost-geodesic_Teichmuller3}
e^{-2\epsilon}\le
\max_{H\in \mathcal{MF}_1}
\frac{i(H,G)}{\ext_{\gamma(s)}(H)^{1/2}}
\max_{H\in \mathcal{MF}_1}\ext_{\gamma(s)}(H)^{1/2}
\le
e^{2\epsilon}
\end{equation}
for $s\ge N$.
%by Lemma \ref{lem:t_0_is_one}.
%since $\mathcal{E}_{\gamma(t)}(\cdot )=
%(\ext_{\gamma(t)}(\cdot )/{K_{\gamma(t)}})^{1/2}$
%converges to $t_0\,i(\cdot ,G)$ uniformly on $\mathcal{MF}_1$.
From Minsky's inequality \eqref{eq:minsky_inequality}
%Theorem 5.1 in \cite{GM},
and Kerckhoff's formula,
we have
$$
\max_{H\in \mathcal{MF}_1}
\frac{i(H,G)}{\ext_{\gamma(s)}(H)^{1/2}}
=
\sup_{H\in \mathcal{MF}-\{0\}}
\frac{\,i(H,G)}{\ext_{\gamma(s)}(H)^{1/2}}
=\ext_{\gamma(s)}(G)^{1/2}
$$
and
$$
%e^{s-\epsilon}
%\le
\max_{H\in \mathcal{MF}_1}\ext_{\gamma(s)}(H)^{1/2}
=K_{\gamma(s)}.
%\le e^{s+\epsilon}.
$$
Hence we get
\begin{equation*}
%\label{eq:almost-geodesic_Teichmuller3}
e^{-2\epsilon}\le
\ext_{\gamma(s)}(G)^{1/2}\cdot 
K_{\gamma(s)}^{1/2}
\le
e^{2\epsilon}
\end{equation*}
for $s\ge N$
from \eqref{eq:almost-geodesic_Teichmuller3}.
This implies \eqref{eq:extremal_length_G_s}.
%Hence,
%we get
%\begin{equation*} \label{eq:almost-geodesic_Teichmuller4}
%%e^{-3\epsilon}e^{-s}
%%\le
%t_0\ext_{\gamma(s)}(G)^{1/2}
%\le e^{3\epsilon}e^{-s}.
%\end{equation*}
%On the other hand,
%from the distortion property,
%$\ext_{\gamma(s)}(G)^{1/2}\ge e^{-\epsilon}e^{-s}$
%holds in general.
%Therefore,
%we have
%\begin{equation} \label{eq:almost-geodesic_Teichmuller5}
%t_0e^{-\epsilon}e^{-s}\le t_0\ext_{\gamma(s)}(G)^{1/2}
%\le
%e^{3\epsilon}e^{-s}
%\end{equation}
%and $t_0\le 1$,
%which is what we wanted.
\end{proof}

Although the following corollary will not be used in the remainder of this paper,
we include it because it helps to understand the asymptotic behavior of almost geodesic rays.
%The following corollary is not used essentially in this paper.
%However,
%it seems to be useful to understand the asymptotic behaviors
%of almost geodesic rays.
%From the proof of the lemma above,
%we also observe the following.

\begin{corollary}
\label{lcoro:limit_is_G}
Let $\gamma:T\to T(X)$ be an almost geodesic ray
with $\gamma(0)=x_0$
which converges to the projective class of $G\in \mathcal{MF}_1$.
We take $G_t\in \mathcal{MF}_1$
such that $R_{G_t,x_0}(d_T(x_0,\gamma(t)))=\gamma(t)$.
Then,
$G_t$ converges to $G$ as $t\to \infty$.
\end{corollary}

\begin{proof}
Let $G_\infty\in \mathcal{MF}_1$ be an accumulation point of $\{G_t\}_{t\in T}$.
From Lemma \ref{lem:t_0_is_one},
we obtain
\begin{align}
i(\alpha,G_\infty)
&=
\lim_{T\ni t\to \infty}i(\alpha,G_t)
\le
\lim_{T\ni t\to \infty}
\ext_{\gamma(t)}(\alpha)^{1/2}
\ext_{\gamma(t)}(G_t)^{1/2}
\nonumber\\
&\le
e^{\epsilon}\lim_{T\ni t\to \infty}
e^{-t}\ext_{\gamma(t)}(\alpha)^{1/2}
\le 
e^{2\epsilon}\lim_{T\ni t\to \infty}
\mathcal{E}_{\gamma(t)}(\alpha)
\nonumber\\
&=e^{2\epsilon}\,i(\alpha,G).
\label{eq:comparizon_G_infty_to_G2}
\end{align}
Since $\epsilon>0$ is taken arbitrary,
we get
\begin{equation}
\label {eq:comparizon_G_infty_to_G}
i(\alpha,G_\infty)\le i(\alpha, G)
\end{equation}
for all $\alpha\in \mathcal{S}$.
Thus,
it follows from the Gardiner's minimal norm property
%(cf. Proposition \ref{prop:minimal_norm_property})
that
$\|J_{G_\infty,x_0}\|\le \|J_{G,x_0}\|$
(See \cite{Gar}. See also Theorem 3.2 of \cite{MS3}).
%\begin{equation}
%\label{eq:t_0_is_greater_than_one}
%1=\ext_{x_0}(G_\infty)=\|J_{G_\infty,x_0}\|\le \|J_{t_0G,x_0}\|
%=\|J_{G,x_0}\|
%=t_0^2,
%\end{equation}
On the other hand,
since
$\|J_{G_\infty,x_0}\|=1=\|J_{G,x_0}\|$
%by the calculation in \eqref{eq:t_0_is_greater_than_one}
and the conclusion from the equality of the minimal norm property,
we get $J_{G_\infty,x_0}=J_{G,x_0}$
and $G_\infty=G$.
\end{proof}

It follows from this corollary
that if a geodesic ray $R_H$ converges to the projective class $[G]$
of $G\in \mathcal{MF}$
in the Gardiner-Masur compactification,
$H$ and $G$ are projectively equivalent.

%Notice from \eqref{eq:almost-geodesic_Teichmuller5}
%and  Lemma \ref{lem:t_0_is_one}
%that
%\begin{equation}
%\label{eq:extremal_length_G_s}
%\lim_{t\to \infty}\|J_{G,\gamma(t)}\|\cdot K_{\gamma(t)}
%=
%\lim_{t\to \infty}\ext_{\gamma(t)}(G)\cdot K_{\gamma(t)}=1.
%\end{equation}
%and
%\begin{equation}
%\label{eq:intersection_number_G_s_G}
%i(G_t,G)\le \ext_{\gamma(s)}(G_t)^{1/2}\ext_{\gamma(t)}(G)^{1/2}
%\le e^{3\epsilon}e^{-2t}
%\end{equation}
%for all $t\ge N$.

\section{Measured foliations with foliated annuli}
In this section,
we give the asymptotic behavior of
moduli of characteristic annuli corresponding to foliated annuli
% in $G$
and
the twisting number of closed geodesics on characteristic annuli.
These observations will be used for proving Theorem \ref{thm:main2}
in the next section.
%Throughout this section,
%we suppose in addition that $G$ has a component of a foliated annulus
%with core $\alpha\in \mathcal{S}$.
%Let $A_t$ be the characteristic annulus of $J_{G,\gamma(t)}$
%corresponding the foliated annulus.

As in \S\ref{subsec:case_where_G_busemann},
we continue to suppose that the projective class $[G]$ of $G\in \mathcal{MF}_1$
is the limit of an almost geodesic ray $\gamma\colon T\to T(X)$.
Throughout this section
we suppose in addition that $G$ has a component which is a foliated annulus
with core $\alpha\in \mathcal{S}$.
Let $w_0$ be the width of the foliated annulus for $\alpha$ in $G$.
%Namely,
%$G$ w_0\alpha+F$ for some $w_0>0$ and $F\in \mathcal{MF}$.
%Let $w_0$ be the height of the foliated annulus in $G$ with respect to $\alpha$.
%In this section,
%we assume in addition that $G$ is rational 
%and let $G=\sum_{i=1}^kw_i\alpha_i$.
For the simplicity,
we set $J_{t}=J_{G,\gamma(t)}$.
Let $\gamma(t)=(Y_t,f_t)$ for $t\in T$
and $A_{t}\subset Y_t$ be the characteristic annulus of $J_t$ for $\alpha$.

\subsection{Moduli of characteristic annuli}
The modulus of $A_t$ behaves asymptotically as follows.
\begin{lemma}
\label{lem:char_annulus_modulus}
%We have 
${\rm Mod}(A_{t})\asymp K_{\gamma(t)}$
%\begin{equation} \label{eq:estimation_modulus}
%%B_1\,K_{\gamma(t)}\le 
%{\rm Mod}(A_{t})\asymp K_{\gamma(t)}
%\end{equation}
as $t\to \infty$.
\end{lemma}

\begin{proof}
%Let $\alpha\in \mathcal{S}$ be the core curve of $A$.
From \eqref{eq:definition_extremal_length_geometric},
$$
{\rm Mod}(A_{t})\le 1/\ext_{\gamma(t)}(\alpha)
\le K_{\gamma(t)}/\ext_{x_0}(\alpha)
$$
for all $t\in T$.
On the other hand,
from Lemma \ref{lem:behavior_of_K},
%\eqref{eq:extremal_length_G_s},
%and Lemma \ref{lem:rough_additivity},
%for any $\epsilon>0$,
%there is an $N$
%such that $t\ge N$,
\begin{align*}
1/{\rm Mod}(A_{t})
&=\ell_{J_{t}}(\alpha)/w_0
=w_0^{-2}(\mbox{$J_{t}$-area of $A_{t}$}) \\
&\le
%\sum_{i=1}^k \ell_{J_{t}}(\alpha_i)/w_i\\
%&\le
%\min\{w_i^2\}^{-1}
%\sum_{i=1}^k w_i\ell_{J_{t}}(\alpha_i) 
%=
w_0^{-2}\|J_{t}\| =w_0^{-2}\ext_{\gamma(t)}(G)
\asymp K_{\gamma(t)}^{-1}
\end{align*}
as $t\to \infty$.
%Thus we get
%$B_1K_{\gamma(t)}\le {\rm Mod}(A_{t})\le B_2K_{\gamma(t)}$
%for sufficiently large $t\in T$
%and $i=1,\cdots,k$,
%where $B_1$ and $B_2$ are constants independent of $t$.
\end{proof}

%\subsection{Almost geodesic ray in Thurston compactification}
%In this section,
%we shall treat almost geodesic rays in the Thurston compactification
%in the case when $G$ is rational.
%Let $\gamma:T\to T(X)$,
%$G_t$ and $G$ be as above.

%\section{Rational foliations in the Gardiner-Masur boundary}
%\section{Asymptotic behavior of twisting numbers}
%We continue to suppose that the projective class $[G]$ of $G\in \mathcal{MF}_1$
%is the limit of an almost geodesic ray $\gamma\colon T\to T(X)$.
%In this section,
%we assume in addition that $G$ is rational 
%and let $G=\sum_{i=1}^kw_i\alpha_i$.
%For the simplicity,
%we set $J_{t}=J_{G,\gamma(t)}$.
%Let $\gamma(t)=(Y_t,f_t)$ for $t\in T$
%and $A_{i,t}$ be the characteristic annulus of $J_t$ for $\alpha_i$.
%
%In the next section,
%we will show the following.
%
%\begin{theorem}[Multicurves are not Busemann points]
%\label{thm:multi-curve_not_limit}
%Let $G=\sum_{i=1}^kw_i\alpha_i\in \mathcal{MF}_1$
%be a maximal rational measured foliation.
%Then,
%the projective class $[G]$ is a Busemann point
%if and only if $k=1$.
%\end{theorem}
%In this section,
%we devote to give asymptotic behavior of
%the twisting number of closed geodesics on the characteristic annuli.
%These observations will be used for proving of Theorem \ref{thm:main2}
%in the next section.
%\ref{thm:multi-curve_not_limit}.
%Notice
%that the conclusion of this section holds for 
%(non-maximal) rational foliations.
%Namely,
%we only assume that $G$ is rational.

\subsection{Twisting numbers of paths in flat annuli}
We here define the \emph{twisting numbers} of proper paths
in flat annuli.
Let $\mathbb{S}^1_L$ be the Euclidean circle of length $L$.
Let $A=[0,m]\times \mathbb{S}^1_L$ be a flat annulus.
Let $\beta\subset A$ be an (unoriented) path connecting
components of $\partial A$.
Take a universal cover $[0,m]\times \mathbb{R}\to A$.
Let $\tilde{\beta}$ be a lift of $\beta$.
Let $(0,y_1),(m,y_2)\in [0,m]\times \mathbb{R}$ be the endpoints of $\tilde{\beta}$.
Then,
we define a \emph{twisting number} ${\rm tw}_A(\beta)$ of $\beta$
in $A$ by
$$
{\rm tw}_A(\beta)=|y_1-y_2|/L.
$$
One can easily check that the twisting number is defined independently of the choice of lifts
(cf. Figure \ref{fig:twisting_number}.
See also \cite{Minsky}).
\begin{figure}
\includegraphics[height=3cm]{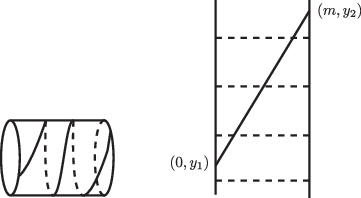}
\caption{Twisting number}
\label{fig:twisting_number}
\end{figure}

\subsection{Twisting numbers of geodesics}
\label{subsec:twisting_numbers}
Let $\beta\in \mathcal{S}$
with $\beta\ne \alpha$.
For $t\in T$,
we let $\beta^*_t$ be a geodesic representative of $\beta$
in $Y_t$ with respect to 
the $J_{t}$-metric.
If $J_t$ admits a flat annulus whose core is homotopic to $\beta$,
we choose one of the closed trajectories in the flat annulus to define $\beta^*_t$.
%Since $G$ is rational,
%$\beta^*_t$ consists of horizontal saddle connections
%and straight saddle connections traversing characteristic annuli of $J_t$.

Let $\beta^*_t\cap A_t=\{\sigma^1_s\}_{s=1}^{n_0}$
be the set of $n_0$-straight line segments of $\beta^*_t$  that lie in $A_t$, counted with multiplicity.
Notice that $n_0=i(\beta,\alpha)$,
which is independent of the parameter $t$.
%be the set of straight segment in $\beta^*_t$ in the part of $A_t$
%counting multiplicity,
%where $n_0=i(\beta,\alpha)$.
Let $\{\sigma^2_j\}_j$ be a collection
of maximal straight segments
in $\beta^*_t\setminus \cup_{s=1}^{n_0}\sigma^1_s$,
counting multiplicity.
%One can easily see that $\sigma^1_s$ is a proper path in the characteristic annulus which
%contains it.
%Notice that
%$$
%n_t\le \sum_{i=1}^k i(\alpha_i,\beta),
%$$
%which is independent of $t\in T$.
In this section,
for a measured foliation $F$ and a path $\sigma$
transverse to the underlying foliation of $F$,
we define
$i(\sigma,F)$
to be the infimum of
the integrals of the transversal measure of $F$ over 
all paths homotopic to $\sigma$ relative to endpoints.

\begin{lemma}[Twisting number]
\label{lem:twisting_number}
%For $s=1,\cdots,n_t$,
%we set $i(s,t)\in \{1,\cdots,k\}$ with
%$\sigma^1_s\subset A_{i(s,t),t}$.
%Then,
For $s=1,\cdots,n_0$,
the twisting number of $\sigma^1_s$ in $A_{t}$ satisfies
$$
{\rm tw}_{A_{t}}(\sigma^1_s)=o(K_{\gamma(t)})
$$
as $t\to \infty$.
\end{lemma}

\begin{proof}
When $n_0=i(\beta,\alpha)=0$,
the geodesic representative $\beta^*_t$
does not intersect the interior of $A_t$.
Hence,
the conclusion automatically holds.
Therefore,
we may assume that $n_0\ne 0$.

Let $q_t=J_t/\|J_t\|$.
%From \eqref{eq:extremal_length_G_s},
Then,
the vertical foliation $V_{q_t}$ of $q_t$ is equal to
$\|J_t\|^{-1/2}\,G$ for all $t\in T$.
In particular,
the $q_t$-height $w_t$ of the characteristic annulus $A_t$
is equal to $\|J_t\|^{-1/2}\, w_0$.
%$(1+o(1))K_{\gamma(t)}^{1/2}\cdot G$
%as $t\to \infty$.
Let $H_{q_t}$ be the horizontal foliation of $q_t$.
Since each $\sigma^1_s$ is a $q_t$-straight segment,
\begin{equation}
\label{eq:vertical_horizontal_q_t}
i(\sigma^j_s,V_{q_t})=\int_{\sigma^j_s}V_{q_t}
\quad
\mbox{and}
\quad
i(\sigma^j_s,H_{q_t})=\int_{\sigma^j_s}H_{q_t}
\end{equation}
for $j=1,2$.
$i(\sigma^j_s,V_{q_t})$ and $i(\sigma^j_s,H_{q_t})$
are called the \emph{horizontal} and \emph{vertical} length of $\sigma^j_s$,
respectively
(cf. \cite{MM1}).

%The following calculation looks technical.
Before giving the details,
we first summarize the following calculation as follows.
From Lemma \ref{lem:behavior_of_K},
the width of the foliated annulus in $V_{q_t}$ of
$\alpha$ is $\|J_t\|^{-1/2}w_0=K_t^{1/2}w_0+o(K_t^{1/2})$
as $t\to \infty$.
Hence,
from Lemma \ref{lem:char_annulus_modulus},
the circumference $\ell_t$
of the characteristic annulus $A_t$
with respect to $q_t$ is comparable with $K_t^{-1/2}$.
By Pythagoras' theorem,
the length of any component $\sigma^j_s$ is the square root
of $K_t^{1/2}w_0+o(K_t^{1/2})$ and the vertical length
$i(\sigma^j_s,H_{J_t})$ in the characteristic annulus.
From those observations and the assumption $\gamma(t)\to [G]$,
we can deduce that the vertical length $i(\sigma^j_s,H_{J_t})$
of $\sigma^j_s$ in the annulus with respect to $J_t$ tends to zero
as \eqref{eq:horozontal_intersection_beta1} below.
Therefore,
from Lemma \ref{lem:behavior_of_K},
the vertical length of $\sigma^j_s$ in the annulus with respect to $q_t$
is $o(K_t^{1/2})$.
Thus,
the twisting number is 
comparable with $i(\sigma^j_s,H_{J_t})/\ell_t=o(K_t)$
as $t\to \infty$ as desired.

Let us start the calculation.
From the notation \eqref{eq:vertical_horizontal_q_t},
by Pythagoras' theorem,
we have
\begin{equation}
\label{eq:length_beta_star_t}
\ell_{q_t}(\beta^*_t)=\sum_{s=1}^{n_0}
\sqrt{i(\sigma^1_s,H_{q_t})^2+i(\sigma^1_s,V_{q_t})^2}
+\sum_j
\sqrt{i(\sigma^2_j,H_{q_t})^2+i(\sigma^2_j,V_{q_t})^2}.
\end{equation}
Since $\|q_t\|=1$,
we have $\ell_{q_t}(\beta^*_t)\le \ext_{\gamma(t)}(\beta)^{1/2}$
from \eqref{eq:definition_extremal_length_analytic}.
Therefore,
\begin{align*}
%(1+o(1))K_{\gamma(t)}^{1/2}
\|J_t\|^{-1/2}i(\beta,G)
&
=
i(\beta,V_{q_t})
\le 
\ell_{q_t}(\beta^*_t) \\
&=
\sum_{s=1}^{n_t}
\sqrt{i(\sigma^1_s,H_{q_t})^2+i(\sigma^1_s,V_{q_t})^2}
%\\
%&\quad\quad
+\sum_j
\sqrt{i(\sigma^2_j,H_{q_t})^2+i(\sigma^2_j,V_{q_t})^2}
\\
&
=\sum_{s=1}^{n_t}
%\sqrt{i(\sigma^1_s,H_{q_t})^2+(1+o(1))K_{\gamma(t)}\cdot
\sqrt{i(\sigma^1_s,H_{q_t})^2+\|J_t\|^{-1}\,
w_0^2}
%\\
%&\quad\quad
+\sum_j
\sqrt{i(\sigma^2_j,H_{q_t})^2+i(\sigma^2_j,V_{q_t})^2}
\\
&\le
\ext_{\gamma(t)}(\beta)^{1/2}.
\end{align*}
%where $w_0$ is the height of the foliated annulus in $G$ with respect to $\alpha$.
%where $\sigma^1_s$ intersects $A_{i(s,t),t}$.
Thus,
we obtain
\begin{align}
%(1+o(1))i(\beta,G)
i(\beta,G)
&\le 
%\sum_{s=1}^{n_t}
%\sqrt{\frac{i(\sigma^1_s,H_{q_t})^2}{K_{\gamma(t)}}+(1+o(1))\cdot w_{s(i,t)}^2} \nonumber\\
%&\quad\quad+\sum_j
%\frac{\sqrt{i(\sigma^2_j,H_{q_t})^2+i(\sigma^2_j,V_{q_t})^2}}{K_{\gamma(t)}^{1/2}} \nonumber \\
\sum_{s=1}^{n_0}
\sqrt{\|J_t\|i(\sigma^1_s,H_{q_t})^2+w_0^2} \nonumber\\
&\quad\quad+\sum_j
\sqrt{\|J_t\|i(\sigma^2_j,H_{q_t})^2+\|J_t\|i(\sigma^2_j,V_{q_t})^2} \nonumber \\
&=\sum_{s=1}^{n_0}
\sqrt{i(\sigma^1_s,H_{J_t})^2+w_0^2}
%\nonumber\\
%&\quad\quad
+\sum_j
\sqrt{i(\sigma^2_j,H_{J_t})^2+i(\sigma^2_j,V_{J_t})^2} \nonumber \\
&\le \|J_t\|^{1/2}\ext_{\gamma(t)}(\beta)^{1/2}.
\label{eq:length_beta_star_t_2}
\end{align}
From the assumption,
Lemma \ref{lem:t_0_is_one}
and \eqref{eq:extremal_length_G_s},
$$
\|J_t\|^{1/2}\ext_{\gamma(t)}(\beta)^{1/2}=
(1+o(1))\ext_{\gamma(t)}(\beta)^{1/2}/K_{\gamma(t)}^{1/2}
=(1+o(1))\mathcal{E}_{\gamma(t)}(\beta)
$$
tends to $i(\beta,G)$ as $t\to \infty$.
Since
$$
i(\beta,G)=\int_{\beta^*_t}V_{J_t}
=
n_0w_0+\sum_ji(\sigma^2_j,V_{J_t}),
%\sum_{s=1}^{n_t}
%\int_{\sigma^1_s}V_{J_t}
%=\sum_{s=1}^{n_t}w_{i(s,t)},
$$
we deduce from \eqref{eq:length_beta_star_t_2} that
the sum
\begin{align}
\sum_{s=1}^{n_0}
&\left(
\sqrt{i(\sigma^1_s,H_{J_t})^2+w_0^2}-w_0
\right)
\label{eq:twist_intersection_number}
\\
&\quad\quad
+\sum_j
\left(
\sqrt{i(\sigma^2_j,H_{J_t})^2+i(\sigma^2_j,V_{J_t})^2}-i(\sigma^2_j,V_{J_t})
\right)
\nonumber
\end{align}
tends to zero as $t\to\infty$.
Since every term in \eqref{eq:twist_intersection_number} is non-negative,
we get
\begin{equation}
\label{eq:horozontal_intersection_beta1}
\lim_{t\to \infty}i(\sigma^1_s,H_{J_t})=0
\end{equation}
for $s=1,\cdots,n_0$.
%\begin{equation}
%\label{eq:horozontal_intersection_beta}
%\lim_{t\to \infty}\frac{1}{K_{\gamma(t)}^{1/2}}
%\left(
%\sum_{s=1}^{n_t}
%i(\sigma^1_s,H_{q_t})+
%\sum_j
%i(\sigma^2_j,H_{q_t})
%\right)
%=0.
%\end{equation}

We now fix $s=1,\cdots,n_0$.
Let $[0,w_t]\times \mathbb{R}\to [0,w_t]\times \mathbb{S}^1_{\ell_t}\cong A_{t}$ be the universal cover,
where $\ell_t$ is the $q_t$-circumference of $A_{t}$.
Let $(0,y_1)$ and $(w_t,y_2)$ be the endpoints of a lift of $\sigma^1_s$.
From the definition,
$$
|y_1-y_2|=i(\sigma^1_s,H_{q_t}).
$$
Since
$$
{\rm Mod}(A_{t})=w_t/\ell_t=\|J_t\|^{-1/2}w_0/\ell_t
=(1+o(1))K_{\gamma(t)}^{1/2}w_0/\ell_t,
$$
%\eqref{eq:estimation_modulus} is equivalent to
from Lemma \ref{lem:char_annulus_modulus},
we obtain
\begin{equation}
\label{eq:behavior_of_ell_t}
\ell_t\asymp K_{\gamma(t)}^{-1/2}
\end{equation}
%\begin{equation}
%\label{eq:estimation_length}
%B'_1/K_{\gamma(t)}^{1/2}
%\le \ell_{i(s,t)}
%\le B'_2/K_{\gamma(t)}^{1/2}
%\end{equation}
for $s=1,\cdots,n_t$.
%where $B'_1$ and $B'_2$ are constants independent of $t\in T$.
Thus,
it follows from \eqref{eq:horozontal_intersection_beta1} that
%$$
%\frac{{\rm tw}_{A_{t}}(\sigma^1_s)}{K_{\gamma(t)}}
%=\frac{|y_1-y_2|/\ell_{t}}{K_{\gamma(t)}}\asymp
%\frac{i(\sigma^1_s,H_{q_t})}{K_{\gamma(t)}^{1/2}}
\begin{align*}
{\rm tw}_{A_{t}}(\sigma^1_s)
&=|y_1-y_2|/\ell_{t}
\asymp
i(\sigma^1_s,H_{q_t})K_{\gamma(t)}^{1/2} \\
&=\|J_t\|^{-1/2}i(\sigma^1_s,H_{J_t})K_{\gamma(t)}^{1/2}
=(1+o(1))i(\sigma^1_s,H_{J_t})K_{\gamma(t)} \\
&=o(K_{\gamma(t)}),
\end{align*}
%$$
which implies what we wanted.
\end{proof}
%For $s=1,\cdots,n_t$,
%the $q_t$-angle $\theta_{s,t}$ of $\sigma^1_s$ satisfies
%$$
%\tan \theta_{s,t}=\frac{i(\sigma^1_s,H_{q_t})}{i(\sigma^1_s,V_{q_t})}.
%$$
%From Lemma \ref{lem:twisting_number},
%we deduce that
%\begin{equation}
%\label{eq:tangent_sigma}
%\tan \theta_{s,t}
%=\frac{\ell_{i(s,t)}\cdot {\rm tw}_{A_{i(s,t),t}}(\sigma^1_s)}{i(\sigma^1_s,V_{q_t})}
%\asymp 
%\frac{{\rm tw}_{A_{i(s,t),t}}(\sigma^1_s)}{K_{\gamma(t)}}
%\to 0
%\end{equation}
%as $t\to \infty$.

%%%%%%%%%%

%\section{Rational Busemann points}
%In this and the next sections,
%we devote to show the following.
%
%\begin{theorem}[Multicurves are not Busemann points]
%\label{thm:multi-curve_not_limit}
%Let $G=\sum_{i=1}^km_i\alpha_i\in \mathcal{MF}_1$
%be a rational measured foliation.
%Then,
%the projective class $[G]$ is a Busemann point
%if and only if $k=1$.
%\end{theorem}

\subsection{Twisting deformations on flat annuli}
\label{subsec:twisting_flat_annuli}
In this section,
we shall recall a canonical quasiconformal mapping of
the twisting deformations along the core curve
on a round annulus (cf. \cite{MardenM}).

Let $A=\{e^{-2\pi m}<|z|<1\}$ be a round annulus of modulus $m$.
For $\tau>0$,
we consider a quasiconformal self-mapping $W_\tau$ of $A$
by
$$
W_\tau(z)=z|z|^{\frac{i\tau}{2\pi m}}.
$$
Notice that $W_\tau$ satisfies
$$
W_\tau\circ \Pi(x+iy)=\Pi\circ L(x+iy)
$$
where $L(x+iy)=x+(\tau/m)y+iy$
and $\Pi:\{x+iy\mid 0\le y\le m\}\to A$ is the universal covering.
Therefore,
the Beltrami differential of $W_\tau$ is equal to
\begin{equation}
\label{eq:bertrami_twist}
\frac{\overline{\partial} W_\tau}{\partial W_\tau}
=\frac{i(\tau/m)}{4\pi+i(\tau/m)}
\frac{z}{\overline{z}}
\frac{d\overline{z}}{dz}.
\end{equation}
%We can check that
We can easily see that 
%\begin{equation} \label{eq:twisting_number_deformation}
%{\rm tw}_A(W_\tau(\sigma))={\rm tw}_A(\sigma)-\tau.
%\end{equation}
%Especially,
when a proper path $\sigma$ in $A$ has the twist parameter $\tau$,
we can choose a sign $s\in \{+1,-1\}$ such that
${\rm tw}_A(W_{s\tau}(\sigma))=0$.

\section{Proof of Theorem \ref{thm:main2}}
% \ref{thm:multi-curve_not_limit}}
In this section,
we shall show Theorem \ref{thm:main2}.
%\ref{thm:multi-curve_not_limit}.
Throughout this section,
we assume that $G=\Sigma_{i=1}^kw_i\alpha_i$ is a maximal rational foliation
and $k=3g-3+m\ge 2$.
Just as in the previous sections,
we assume the projective class $[G]$
is the limit of an almost geodesic ray $\gamma:T\to T(X)$,
and we continue to use the same notation.

%The `if'-part follows from Theorem 7.1 of \cite{GM}.
%Hence,
%we shall show the `only if'-part.

%\subsection{The case of one dimenisional Teichm\"uller space}
%Notice that when $X$ is either a once puntured torus or
%a four punctured sphere,
%the Teichm\"uller space equipped with the Teichm\"uller distance
%is isometric to the hyperbolic disk.
%Hence,
%the Gardiner-Masur compactification coincides with the Gromov-compactification
%of the hyperbolic disk
%and every boundary point is a Busemann point
%(see e.g. \cite{Mi0}).
%On the other hand,
%in this case,
%every rationral foliation consists of one simple closed curve.
%Therefore,
%we have nothing to do for this case.

\subsection{Notation}
Let $A_{i,t}\subset Y_t$ be the characteristic annulus of $q_t=J_t/\|J_t\|$
for $\alpha_i$.
Let $\Sigma_t$ be the critical graph of $q_t$ and
consider the $K_{\gamma(t)}^{-1/2}$-neighborhood
$N_t$ of $\Sigma_t$ in $Y_t$
with respect to the $q_t$-metric.
Let $A^0_{i,t}=A_{i,t}\setminus N_t$
(cf. Figure \ref{fig:foliation}).
%%%
\begin{figure}
\includegraphics[height=5.5cm]{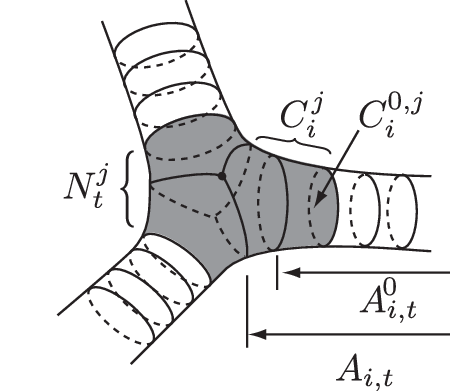}
\caption{The vertical foliation of $q_t$ around a typical component
of $N_t$.
The component $N^j_t$
and annuli $C^j_i$,
$C^{0,j}_i$ in the figure are taken and used in \S\ref{subsubsec:B-j-s}.}
\label{fig:foliation}
\end{figure}
Since the $q_t$-height of $A_{i,t}$ is $(1+o(1))K_{\gamma(t)}^{1/2}w_i$,
when $t\in T$ is sufficiently large,
$A^0_{i,t}$ is a well-defined foliated subannulus of $A_{i,t}$ with height
\begin{equation} \label{eq:width-q_t-A_it}
w''_{i,t}:=(1+o(1))K_{\gamma(t)}^{1/2}w_i-2K_{\gamma(t)}^{-1/2}
%=K_{\gamma(t)}^{1/2}
%((1+o(1))w_i-2/K_{\gamma(t)})
\asymp K_{\gamma(t)}^{1/2}.
\end{equation}

Let $\partial A^0_{i,t}=\partial_1A^0_{i,t}\cup \partial_2A^0_{i,t}$.
Since $\partial A^0_{i,t}$ consists of closed leaves in $A_{i,t}$
and the heights of the remaining annuli in $A_{i,t}-A^0_{i,t}$ 
are at most $K_{\gamma(t)}^{-1/2}$,
from \eqref{eq:behavior_of_ell_t},
the moduli of remaining annuli in $A_{i,t}-A^0_{i,t}$ are uniformly bounded,
and hence
$$
{\rm Mod}(A_{i,t})={\rm Mod}(A^0_{i,t})+O(1)
$$
as $t\to \infty$.

%\subsubsection{The case where $i(\beta,G)=0$}
%In this case,
%by definition,
%$\mathcal{E}_{\gamma(t)}(\beta)$ tends to $i(\beta,G)=0$ as $t\to \infty$.
%Equivalently,
%we have
%\begin{equation} \label{eq:extremal_disjoint_1}
%\ext_{\gamma(t)}(\beta)=o(K_{\gamma(t)})
%\end{equation}
%as $t\to \infty$.
%%Furthermore,
%when $\beta$ is homotopic into a component of $N_t$ which is not a pair of pants
%as a non-peripheral curve,
%the extremal length of $\beta$ does not decrease too much in the sense that
%\begin{equation} \label{eq:extremal_disjoint_2}
%\ext_{\gamma(t)}(\beta)\cdot K_{\gamma(t)}\to \infty.
%\end{equation}
%Indeed,
%we take $\beta'\in \mathcal{S}$
%such that $i(\beta',\beta)\ne 0$ but $i(\beta',G)=0$.
%Since
%$$
%0\ne i(\beta',\beta)^2
%\le \ext_{\gamma(t)}(\beta')\cdot \ext_{\gamma(t)}(\beta)
%=\mathcal{E}_{\gamma(t)}(\beta')^2
%\cdot
%(\ext_{\gamma(t)}(\beta)\cdot K_{\gamma(t)}),
%$$
%we deduce \eqref{eq:extremal_disjoint_2}
%from \eqref{eq:extremal_disjoint_1}.

\subsection{Calculation of extremal length: Lower bound}
Take $\beta\in \mathcal{S}$.
We devote this section to bound the extremal length of $\beta$ from below.
Henceforth,
we suppose that $i(\beta,G)\ne 0$.

%We first give the lower estimate of the extremal length of $\beta$ for $\gamma(t)$.
%The following calculation is similar to that given in \cite{Ker} (see also Appendix of \cite{Mi1}).
Let $\mathcal{A}_t^\beta$ be the characteristic annulus of
the Jenkins-Strebel differential
$J_{\beta,\gamma(t)}$
for $\beta$.
Fix $i=1,\cdots,k$.
The intersection $\mathcal{A}_t^\beta\cap A^0_{i,t}$
contains at least
$n_i=i(\beta,\alpha_i)$ components $\{D^i_l\}_{l=1}^{n_i}$
such that 
$D^i_l$ contains a path connecting $\mathcal{A}_t^\beta\cap
\partial_1A^0_{i,t}$ and $\mathcal{A}_t^\beta\cap \partial_2A^0_{i,t}$.
%$\partial D^i_l$ contains components of $\mathcal{A}_t^\beta\cap
%\partial_1A^0_{i,t}$ and $\mathcal{A}_t^\beta\cap \partial_2A^0_{i,t}$.

Let $\Gamma(D^i_l)$ be the family of rectifiable curves in $D^i_l$
connecting $\mathcal{A}_t^\beta\cap
\partial_1A^0_{i,t}$ and $\mathcal{A}_t^\beta\cap \partial_2A^0_{i,t}$.
%Let $\rho^i_{l,t}$ be the restriction of $|q_t|$ to $D^i_l$,
%and set $\rho^i_t=\sum_{l=1}^{n_i}\rho^i_{l,t}$.
Let $\rho^i_{t}$ be the restriction of the $q_t$-metric
to $D^i=\cup_{i=1}^{n_i}D^i_l$.
From \eqref{eq:width-q_t-A_it},
any curve in $\sum_{i=1}^{n_i}\Gamma(D^i_l)$ has $\rho^i_{t}$-length at most $n_iw''_{i,t}$.
Since the critical graph of
the Jenkins-Strebel differential of $\beta$ on $Y_t$ has measure zero,
$$
A(\rho^i_t)=\iint_{D^i}\rho^i_t\le 
(\mbox{$|q_t|$-area of $A^0_{i,t}$})=\ell_{i,t}w''_{i,t}.
$$
%the totality of the $\rho^i_{l,t}$-area $\mathcal{A}(\rho^i_{l,t})$
%of $D^i_l$ over $l$
%is at most the area $\ell_{i,t}w''_{i,t}$  of $A^0_{i,t}$ in $|J_t|$.
By the definition of extremal length,
we have
\begin{align*}
\ext
\left(
\sum_{i=1}^{n_i}\Gamma(D^i_l)
\right)
&\ge 
L_{\rho^i_{t}}\left(
\sum_{i=1}^{n_i}\Gamma(D^i_l)
\right)
^2/A(\rho^i_{t})
\ge (n_iw''_{i,t})^2/\ell_{i,t}w''_{i,t} \\
&=n_i^2{\rm Mod}(A^0_{i,t})=n_i^2{\rm Mod}(A_{i,t})+O(1).
\end{align*}
%where $\ext(\Gamma(D^i_l))$ is the extremal length of the curve family $\Gamma (D^i_l)$
%(cf. \cite{Ahlfors}).
Since any non-trivial simple closed curve in $\mathcal{A}_t^\beta$
traverses each $D^i_l$ between $\mathcal{A}_t^\beta\cap \partial_1 A^0_{i,t}$
and $\mathcal{A}_t^\beta\cap \partial_2 A^0_{i,t}$,
such simple closed curve contains a curve in $\sum_{i=1}^{n_i}\Gamma(D^i_l)$.
Therefore,
from (3) of Proposition \ref{prop:extremal_length},
we conclude
\begin{equation} \label{eq:extremal_length_lambda_1}
\ext_{\gamma(t)}(\beta)\ge \sum_{i=1}^k
\ext
\left(
\sum_{i=1}^{n_i}\Gamma(D^i_l)
\right)
\ge \sum_{i=1}^{k}n_i^2{\rm Mod}(A_{i,t})+O(1)
\end{equation}
as $t\to \infty$.

\subsection{Calculation of extremal length: Upper bound}
\label{subsec:upper-estimate}
Before discussing the upper bound,
we deform $Y_t$ slightly as follows.
For $i=1,\cdots,k$,
we fix a component $\sigma^1_{s_i}$ of $\beta^*_t\cap A_{i,t}$.
We put the Beltrami differential \eqref{eq:bertrami_twist}
on each flat annulus $A_{i,t}$ with $\tau=\pm {\rm tw}_{A_{i,t}}(\sigma^1_{s_i})$
(we choose the appropriate sign so that the following holds).
We extend the Beltrami differential to $Y_t$ by putting $0$ on the complement.
Then,
we obtain a quasiconformal deformation of $Y_t$ with respect to the Beltrami differential
to get $\gamma'(t)\in T(X)$.
By Lemmas \ref{lem:char_annulus_modulus} and \ref{lem:twisting_number},
\begin{equation}
\label{eq:no-twists}
{\rm tw}_{A_{i,t}}(\sigma^1_{s_i})/{\rm Mod}(A_{i,t})\to 0
\end{equation}
as $t\to\infty$ for all $i$, and hence,
$$
d_T(\gamma(t),\gamma'(t))\to 0
$$
when $t\to\infty$.
One can easily see that $\gamma':T\to T(X)$
is an almost geodesic ray.
Furthermore,
by Proposition 4.9 of \cite{Rieffel},
$\gamma'(t)$ has the same limit as that of $\gamma(t)$
in the Gardiner-Masur compactification.
Thus,
to simplify of the notation,
we may suppose that $\gamma'(t)=\gamma(t)$.

%Notice from \eqref{eq:twisting_number_deformation}
As remarked in \S\ref{subsec:twisting_flat_annuli},
if we choose the sign of $\tau$ suitably,
after this deformation,
the twist parameter of each $\sigma^1_{s_i}$ is zero.
Hence,
any segment in $\beta^*_t\cap A_{i,t}$ has the twisting number at most one
in $A_{i,t}$ for all $i$,
because $\beta$ is a simple closed curve
and any two segments in $\beta^*_t\cap A_{i,t}$
do not intersect transversely in $A_{i,t}$.
By taking a subray,
we may assume that there is a (non-connected) graph $\Sigma_0$ on $X$ such that
the marking $f_t:X\to Y_t$ induces an isomorphism $\Sigma_0$ and $\Sigma_t$
(in homotopy sense).
%Let $\Sigma^j_t$ be a component of the critical graph $\Sigma_t$ in $N^j_t$.

\subsubsection{The idea for getting an appropriate upper bound}
\label{subsec:idea}
To bound of the extremal length from above,
from \eqref{eq:definition_extremal_length_geometric},
it suffices to construct a suitable annulus $\mathcal{A}_t$ on $Y_t$
whose core is homotopic to $f_t(\beta)$.
%The idea for giving the upper bounds is to construct an appropriate annulus
%$\mathcal{A}_t$ for given simple closed curve $\beta$.
The procedure given here is originally due to S. Kerckhoff in \cite{Ker},
when a given almost geodesic ray $\gamma$ is actually a geodesic (See also \S9 of \cite{Mi1}).
We briefly recall the construction in the case where $\gamma$ is a geodesic.
We first cut each characteristic annulus $A_{i,t}$ of $J_{t}$
into $n_i=i(\beta,\alpha_i)$ congruent horizontal rectangles.
The annulus $\mathcal{A}_t$
is made
by composing
appropriately such (slightly modified)
$n_i$ congruent horizontal rectangles
and ties (quadrilaterals) in $N_t$
(cf. \eqref{eq:annulus_model} and \S\ref{subsubsec:B-j-s}).
We can take ties whose extremal lengths are uniform
(cf. Lemma \ref{lem:B_s}).
Then,
by applying Proposition \ref{prop:extremal_length_upper},
we obtain an upper bound of the extremal length of $\mathcal{A}_t$.
%and hence the upper bound of the extremal length of $\beta$
%as $t\to \infty$,
%by the Schwarz lemma for the extremal lengths.
%After obtaining the upper bound,
%we can calculate the limit of $\mathcal{E}_{\gamma(t)}(\beta)$
%and get the contradiction.

A basic reason why we can get an appropriate upper bound in the case above
is that,
along the Teichm\"uller ray associated to the projective class of
$G=\sum_{i=1}^kw_i\alpha_i$,
the characteristic annuli of the Jenkins-Strebel differential for $G$
are deformed with `no-twisting' deformations,
because the Teichm\"uller deformation is accomplished by stretching
in the horizontal and vertical directions.
In the upper bound,
the major part comes from
the extremal length of congruent rectangles
%which defined by dividing
%the characteristic annuli
(cf. \eqref{eq:upper_bound_fini}).
The `no-twisting' property implies that
the totality of the extremal lengths of such rectangles 
coincides with the major part of the lower estimate \eqref{eq:extremal_length_lambda_1}
(see \eqref{eq:extremal_R_i_l}).

In the case where $\gamma$ is an almost geodesic ray,
we have already observed
that $\beta$ is not twisted very much in the characteristic annuli
(cf. \eqref{eq:no-twists}).
Hence,
we can apply the similar argument for getting an appropriate upper bound
of $\ext_{y_t}(\beta)$.

%\subsubsection{Case : $G$ is maximal}
%\label{subsubsec:case_G_is_maximal}
\subsubsection{Ties $\{B^j_s\}_s$}
\label{subsubsec:B-j-s}
%We first assume that $G$ is maximal,
%that is,
%the support of $G$ is a pants decomposition on $X$.
%In this case,
In accordance with the idea explained above,
we shall construct appropriate ties in pairs of pants $N_t^j$.

Since $G$ is maximal,
any component $N^j_t$ ($j=1,\cdots, 2g-2+m$)
of $N_t$ is one of the three types:
a pair of pants,
an annulus with one distinguished point (a singularity of angle $\pi$ or a flat point),
or
a half-pillow with two cone singularities of angle $\pi$ (cf. Figure \ref{fig:pants}).
We now assume that $N^j_t$ is a pair of pants
because
the case where $N^j_t$ is an annulus or a half-pillow
can be dealt with in the same manner.

Notice from \eqref{eq:behavior_of_ell_t} that
the length of any component of $\partial N^j_t$ is of order
$K_{\gamma(t)}^{-1/2}$ with respect to the metric $\rho^j_t:=|q_t|_{N^j_t}$.
For simplifying of the notation,
we assume that components of $\partial N^j_t$ are $\alpha_{i_1}$,
$\alpha_{i_2}$, and $\alpha_{i_3}$.
Then,
the critical graphs $\Sigma_t\cap N^j_t$ forms one of the graph in Figure \ref{fig:pants}
(cf. \cite{FLP}).
%We label each edges $e_{ij}$ as Figure \ref{fig:pants}.
%%%
\begin{figure}
\includegraphics[height=4cm]{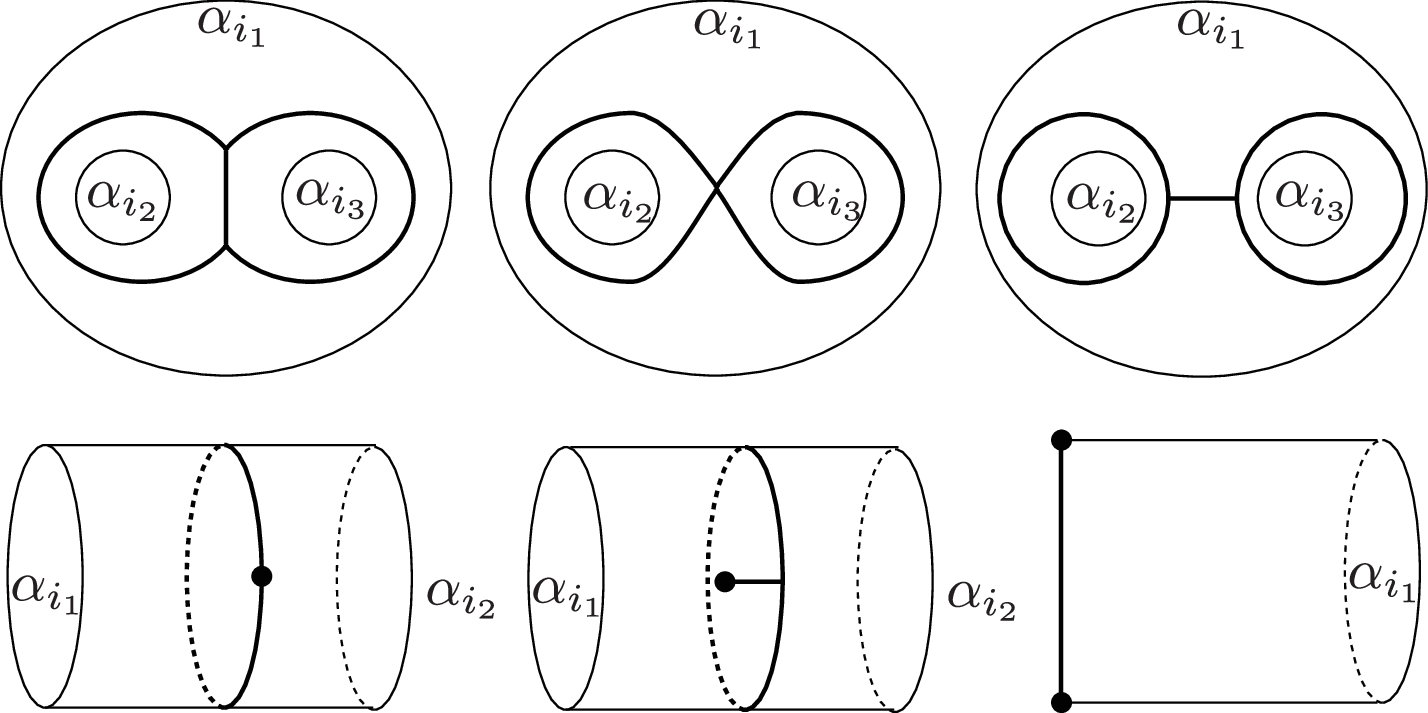}
\caption{Graphs in $N^j_t$.}
\label{fig:pants}
\end{figure}

We make equally spaced $n_{i_l}$ cuts in $\alpha_{i_l}$
where $n_{i_l}=i(\beta,\alpha_{i_l})$ ($l=1,2,3$).
Let $C^j_{i_l}$ be a component of $N^j_t\setminus \Sigma_t$
which contains $\alpha_{i_l}$ in the boundary.
Let $C^{0,j}_{i_l}$ be a subannulus of $C^j_{i_l}$ with height $(2K_{\gamma(t)})^{-1/2}$
and $\alpha_{i_l}\subset C^{0,j}_{i_l}$.
We cut $C^{0,j}_{i_l}$ along the vertical slits with endpoints
in the $n_{i_l}$-cuts in $\alpha_{i_l}$
and get a family of Euclidean rectangles.
Since the circumference and the height of $C^{0,j}_{i_l}$
are of order $(K_{\gamma(t)})^{-1/2}$,
the moduli of such Euclidean rectangles are uniformly bounded above and below.
%Since the intersection numbers $\{n_{i_l}\}_{l=1}^3$
%and the moduli of $\{C^j_{i_l}\}_{l=1}^3$ are uniformly bounded, 
%%%
\begin{figure}
\includegraphics[height=3.5cm]{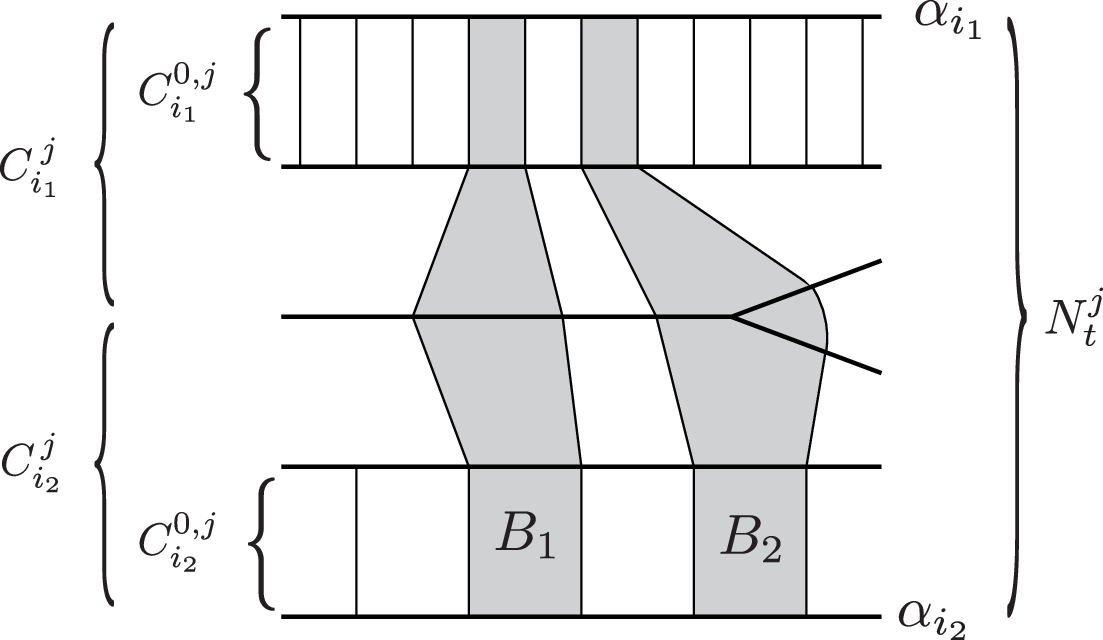}
\caption{Typical examples of ties $B^j_1$ and $B^j_2$ in $N^j_t$.
$B^j_1$ is a `regular' quadrilateral and $B^j_2$ is a `singular' quadrilateral.}
\label{fig:ties}
\end{figure}

The following lemma asserts that
the contribution of the ties is small
for the extremal length of the model annulus
which will be constructed in \eqref{eq:annulus_model} later.

\begin{lemma}[See Figure \ref{fig:ties}]
\label{lem:B_s}
There is a family $\{B^j_s\}_s$ of (singular) quadrilaterals
such that 
\begin{itemize}
\item[(1)]
$B^j_s\cap C^{0,j}_{i_l}$ is a rectangle above for all $s$ and $l$,
\item[(2)]
the arc system given by correcting cores of $B^j_s$'s is homotopic to $\beta\cap N^j_t$,
where the \emph{core} of $B^j_s$ is a path in $B_s$ connecting
facing arcs in
$B^j_s\cap \partial N^j_t$
and
\item[(3)]
the extremal length of the family of paths in $B^j_s$ homotopic to the core
is uniformly bounded above.
\end{itemize}
\end{lemma}

%\begin{proof}[Proof of Lemma \ref{lem:B_s}]
\begin{proof}
Notice from \eqref{eq:behavior_of_ell_t} and
the uniformity of the moduli of $\{C^j_{i_l}\}_{l=1}^3$
that the conformal structure of $N^j_t$ is precompact in
the reduced Teichm\"uller space of three holed spheres
(cf. \cite{Earle}).
Since the intersection numbers $\{n_{i_l}\}_{l=1}^3$ are independent of $t$,
we can take $B^j_s$
% each rectangles in $\cup_{l=1}^3C^{0,j}_{i_l}$
%appropriately as Figure \ref{fig:ties}
such that
%Indeed,
%we can take $B^j_s$
%such that 
the width of each $B^j_s$ with respect to the $q_t$-metric
are comparable with $K_{\gamma(t)}^{-1/2}$.
%We restrict the $q_t$-metric on each $B^j_s$.
By definition,
the $|q_t|$-area of each $B^j_s$ is $O(K_{\gamma(t)}^{-1})$.

From the reciprocal relation between the module and the extremal length for quadrilateral
or Rengel's type inequality,
the extremal length $\ext (B^j_s)$ of the family of paths in
$B^j_s$ homotopic to the core
satisfies
\begin{equation}
\label{eq:extremal_length_B_s}
%\frac{(\mbox{$|q_t|$-height})^2}{\mbox{$|q_t|$-area}}
%\le
\ext (B^j_s)
\le
\frac{\mbox{$|q_t|$-area}}{(\mbox{$|q_t|$-width})^2}
=O(1)
\end{equation}
for all $s$
(see \S4 in Chapter I of \cite{LV}).
\end{proof}

\subsubsection{Construction of a model $\mathcal{A}_t$ of the extremal annulus}
We divide each $A_{i,t}$ into congruent $n_i=i(\beta,\alpha_i)$ rectangles
$\{R_{i,l}\}_{l=1}^{n_i}$ via proper horizontal segments.
We may assume that for any $l$ and $j$,
there is an $s$ such that
$R_{i,l}\cap C^{0,j}_i$
is congruent to $B^j_s\cap C^{0,j}_i$.
We set
$$
R^0_{i,l}=R_{i,l}\setminus (N_t\setminus \cup_j C^{0,j}_i)
$$
(cf. Figure \ref{fig:annulusR}).
%%%
\begin{figure}
\includegraphics[height=4.5cm]{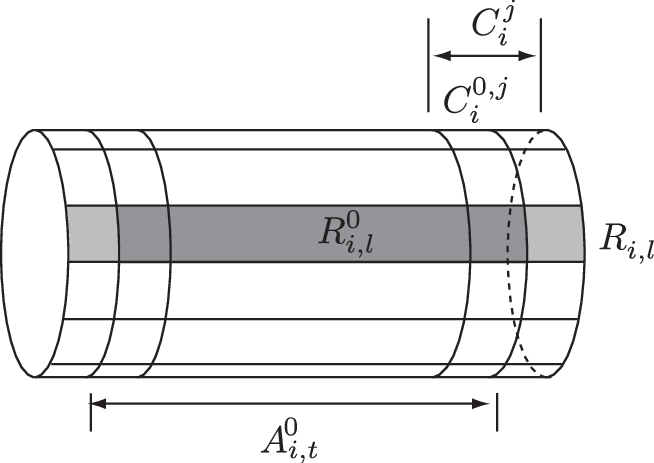}
\caption{Rectangles $R_{i,l}$ and $R^0_{i,l}$ in the annulus $A_{i,t}$.}
\label{fig:annulusR}
\end{figure}
Since twisting numbers of segments in $\beta^*_t\cap A_{i,t}$
on each $A_{i,t}$ are at most one for all $i$,
from \eqref{eq:bertrami_twist} and
the Dehn-Thurston's parametrization of simple closed curves
(cf. \cite{FLP}),
we can glue all $A_{i,t}$ and $N_t$ appropriately at the part $C_i^{0,j}$
to get a Riemann surface $Y'_t$ and an annulus
\begin{equation} \label{eq:annulus_model}
\mathcal{A}'_t=
(\cup_{i,l}R^0_{i,l})
\cup
(\cup_{j,s}B^j_s).
\end{equation}
Since the moduli of the characteristic annuli diverge,
after deforming $Y'_t$ by a quasiconformal mapping with maximal dilatation $1+o(1)$,
we obtain $Y_t$ and the core of the
image $\mathcal{A}_t$ of the annulus $\mathcal{A}'_t$
is homotopic to $f_t(\beta)$.
%(cf. \eqref{eq:bertrami_twist}).
Thus,
we conclude
\begin{equation}
\label{eq:extremal_length_estimate1}
\ext_{y_t}(\beta)=\ext_{Y_t}(f_t(\beta))\le \ext (\mathcal{A}_t)
=(1+o(1))\ext(\mathcal{A}'_t)
\end{equation}
as $t\to \infty$.
Thus,
to get the upper estimate of the extremal length of $\beta$ on $y_t$,
it suffices to give an upper estimate of the extremal length of $\mathcal{A}'_t$.

\subsubsection{Estimate of extremal length of $\mathcal{A}'_t$}
Let $\rho^{\mathcal{A}}_t$ be the extremal metric on $\mathcal{A}'_t$
for the extremal length $\ext (\mathcal{A}'_t)$
with $A(\rho^{\mathcal{A}}_t)=1$.
Let $\{\mathcal{S}_u\}_{u}$ be a collection of all rectangles of
the form $C^{0,j}_i\cap B_s^j$ for all $i,j,s$.
By the same argument as Claim 1 in \S 9.6 in \cite{Mi1},
we can see the following.

\begin{claim}
\label{claim:upper_bound_1}
For any $u$,
there is a vertical line $\eta'_u$ in $\mathcal{S}_u$
such that
$$
\sum_{u}\ell_{\rho^{\mathcal{A}}_t}(\eta'_u)=O(1)
$$
as $t\to \infty$.
\end{claim}
Let us continue the calculation.
Let $\{D_{u}\}_{u}$
be a collection of components of $\mathcal{A}'_t\setminus \cup_u\eta'_u$.
By labeling correctly,
we may assume that
$\partial D_u$ contains $\eta'_u$ and $\eta'_{u+1}$,
where $\eta'_u$ is labeled cyclically in $u$.
By definition,
each $D_{u}$ is contained in either $R_{i,l}$ or $B^j_s$ for some $i,l,j,s$.
Let $\Gamma (D_{u})$ be the family of rectifiable paths connecting vertical segments 
$\eta'_u$ and $\eta'_{u+1}$.
Let $\Gamma(R_{i,l})$ be the family of rectifiable paths in $R_{i,l}$ connecting vertical boundary segments.
Since
% the $q_t$-height of $A_{i,t}$ is $(1+o(1))K^{1/2}_{\gamma(t)}w_i$,
\begin{equation}
\label{eq:extremal_R_i_l}
\ext (\Gamma(R_{i,l}))=(\mbox{the height of $A_{i,t}$})/(\ell_{i,t}/n_i)
=n_i{\rm Mod}(A_{i,t}),
\end{equation}
by (1) and (2) of Proposition \ref{prop:extremal_length_upper} and Claim \ref{claim:upper_bound_1},
we have
\begin{align}
\ext(\mathcal{A}'_t)^{1/2}
&=
\left(
\sum_{u}\ext (\Gamma (D_{u}))
\right)^{1/2}+O(1) \nonumber \\
&\le
\left(
\sum_{D_u\subset R_{i,l}}\ext (\Gamma(R_{i,l}))
+
\sum_{D_u\subset B^j_s}\ext (B^j_s)
\right)^{1/2}+O(1) \nonumber \\
&\le
\left(
\sum_{D_u\subset R_{i,l}}
n_i{\rm Mod}(A_{i,t})
+
O(1)
\right)^{1/2}+O(1) \nonumber \\
&=
\left(
\sum_{i=1}^k
n_i^2{\rm Mod}(A_{i,t})
+
O(1)
\right)^{1/2}+O(1).
\label{eq:upper_bound_fini}
\end{align}
Thus we get the desired upper bound of
the extremal length of $\mathcal{A}'_t$,
and hence of $\ext_{y_t}(\beta)$ from
\eqref{eq:extremal_length_estimate1}.

\subsection{Conclusion}
%\subsubsection*{Conclusion when $G$ is maximal}
%By \eqref{eq:extremal_length_lambda_1} and \eqref{eq:upper_bound_fini},
%we have
%$$
%\left(
%\sum_{i=1}^k
%n_i^2{\rm Mod}(A_{i,t})
%\right)^{1/2}+O(1)
%\le
%\ext_{y_t}(\beta)^{1/2}
%\le
%\left(
%\sum_{i=1}^k
%n_i^2{\rm Mod}(A_{i,t})
%+
%O(1)
%\right)^{1/2}+O(1).
%$$
From Lemma \ref{lem:char_annulus_modulus},
by taking a subray if necessary,
we may assume that ${\rm Mod}(A_{i,t})/K_{\gamma(t)}$ tends to
a positive number
$M_i$ for each $i=1,\cdots,k$.
From
%\eqref{eq:extremal_disjoint_1},
\eqref{eq:extremal_length_lambda_1},
\eqref{eq:extremal_length_estimate1},
and \eqref{eq:upper_bound_fini},
we deduce that
\begin{equation}
%\begin{align}
%\sum_{i=1}^kw_ii(\beta,\alpha_i)
%&=
i(\beta,G)
%&
%=\lim_{t\to \infty}\mathcal{E}_{\gamma(t)}(\beta)
=
\lim_{t\to \infty}
\left(
\frac{\ext_{\gamma(t)}(\beta)}{K_{\gamma(t)}}
\right)^{1/2}
%\nonumber \\
%&
=
\left(
\sum_{i=1}^k
n_i^2M_i
\right)^{1/2}
=
\left(
\sum_{i=1}^k
M_i i(\beta,\alpha_i)^2
\right)^{1/2}
\label{eq:quadratic_intersection}
%\end{align}
\end{equation}
for all $\beta\in \mathcal{S}$.
Since the set $\mathbb{R}_+\otimes \mathcal{S}$
of the weighted simple closed curves
is dense in $\mathcal{MF}$,
and the intersection number function is continuous
and homogeneous
on the product $\mathcal{MF}\times \mathcal{MF}$,
the above equation \eqref{eq:quadratic_intersection}
still holds for all measured foliations $\beta\in \mathcal{MF}$.
Thus,
for any $x,y>0$ and $\beta_1,\beta_2\in \mathcal{S}$ with $i(\beta_1,\beta_2)=0$,
by substituting
$\beta=x\beta_1+y\beta_2$ to \eqref{eq:quadratic_intersection},
we get
\begin{align*}
\left(
\sum_{i=1}^k
M_i n_{1,i}^2
\right)x^2
&+2
\left(
\sum_{i=1}^k M_in_{1,i}n_{2,i}
\right)xy+
\left(
\sum_{i=1}^k
M_i n_{2,i}^2
\right)y^2 \\
&=
(xi(\beta_1,G)+yi(\beta_2,G))^2,
\end{align*}
where $n_{j,i}=i(\beta_j,\alpha_i)$.
This means that
the discriminant of the quadratic form above is zero.
Namely,
we have
$$
\left(
\sum_{i=1}^k M_in_{1,i}n_{2,i}
\right)^2-
\left(
\sum_{i=1}^k
M_i n_{1,i}^2
\right)
\left(
\sum_{i=1}^k
M_i n_{2,i}^2
\right)
=0,
$$
for all such pair of curves $\beta_1,\beta_2\in \mathcal{S}$.
Hence,
two vectors
$$
(\sqrt{M_1}n_{1,1},\cdots,\sqrt{M_k}n_{1,k}),\quad
(\sqrt{M_1}n_{2,1},\cdots,\sqrt{M_k}n_{2,k}),
$$
are parallel for all $\beta_1,\beta_2\in \mathcal{S}$
with $i(\beta_1,\beta_2)=0$.
However,
this is impossible when $k=3g-3+m\ge 2$,
as we already observed
in Section 7 of \cite{Mi1}.

\end{document}